\documentclass[12pt,reqno]{amsart}
\usepackage{graphicx}
\usepackage{amssymb}
\usepackage{amsfonts}
\usepackage{amsmath,latexsym,amssymb,amsfonts,amsbsy, amsthm}
\usepackage[usenames]{color}
\usepackage{xspace,colortbl}
\usepackage{mathrsfs}
\usepackage[colorlinks,linkcolor=blue,citecolor=blue]{hyperref}
\usepackage[titletoc]{appendix}
\usepackage{pdfsync}
\usepackage{esint}
\allowdisplaybreaks

\newcommand{\beq}{\begin{equation}}
	\newcommand{\eeq}{\end{equation}}
\newcommand{\ben}{\begin{eqnarray}}
	\newcommand{\een}{\end{eqnarray}}
\newcommand{\beno}{\begin{eqnarray*}}
	\newcommand{\eeno}{\end{eqnarray*}}

\newtheorem{thm}{Theorem}[section]
\newtheorem{defi}[thm]{Definition}
\newtheorem{lem}[thm]{Lemma}
\newtheorem{prop}[thm]{Proposition}
\newtheorem{coro}[thm]{Corollary}
\newtheorem{rmk}[thm]{Remark}

\def\XXint#1#2#3{{\setbox0=\hbox{$#1{#2#3}{\int}$ }
		\vcenter{\hbox{$#2#3$ }}\kern-.6\wd0}}

\title[Quantitative stratification]{Quantitative stratification for the  fractional Allen-Cahn equation and stationary nonlocal minimal surface}

\author[K. Wang]{Kelei Wang$^\dag$}
\address{$^\dag$School of Mathematics and Statistics \\ Wuhan University\\
	Wuhan 430072, China}
\email{wangkelei@whu.edu.cn}

\author[J. Wei]{Juncheng Wei$^\ast$}
\address{$^\ast$Department of Mathematics \\ Chinese University of Hong Kong\\
	Shatin, NT, Hong Kong}
\email{wei@math.cuhk.edu.hk}

\author[K. Wu]{Ke Wu$^\ast$}
\address{$^\ast$Department of Mathematics \\ Chinese University of Hong Kong\\
	Shatin, NT, Hong Kong}
\email{wke@math.cuhk.edu.cn}

\thanks{K. Wang is supported by  National Key R\&D Program of China (No. 2022YFA1005602) and the National Natural Science Foundation of China (No. 12131017 and No. 12221001).  J. Wei is partially supported  by  National Key R\&D Program of China (No. 2022YFA1005602), and Hong Kong General Research Fund "New frontiers in singular limits of nonlinear partial differential equations. K. Wu is supported by the  National Natural Science Foundation
of China ( No. 12401264). We sincerely thank Wei Wang for numerous valuable discussions.}



\begin{document}

\begin{abstract}
We study properties of solutions to the fractional Allen-Cahn equation  when $s\in (0, 1/2)$ and dimension $n\geq 2$. By applying the quantitative stratification principle developed by Naber and Valtorta, we obtain an optimal quantitative estimate on the transition set. As an application of this estimate, we improve the potential energy estimates of Cabre, Cinti, and Serra \cite{Cabre-Cinti-Serra2021}, providing sharp versions for the fractional Allen-Cahn equation. Similarly, we obtain optimal perimeter estimates for stationary nonlocal minimal surfaces, extending previous results of Cinti, Serra, and Valdinoci \cite{Cinti-Serra-Valdinoci2019} from the stable case.

\end{abstract}

\maketitle

\tableofcontents


\section{Introduction}\label{sec introduction}
	\setcounter{equation}{0}
\subsection{Background} In this paper, we consider the fractional Allen-Cahn equation
\begin{equation}\label{fractional Allen Cahn equation}
(-\Delta)^{s}u_{\epsilon}+W'(u_{\epsilon})=0,\quad\text{in}~ B_{R}(0)
\end{equation}
subject to the exterior Dirichlet boundary condition
\begin{equation}\label{Dirichlet boundary condition}
u_{\epsilon}=g_{\epsilon},\quad\text{on}~\mathbb{R}^{n}\backslash B_{R}(0),
\end{equation}
where $n\geq 2, \epsilon\in (0, 1), B_{R}(0)=\{x\in\mathbb{R}^{n}:|x|\leq R\}$ and $g_{\epsilon}:\mathbb{R}^{n}\to\mathbb{R}$ is a given bounded smooth function. For $ s\in (0, 1/2)$,  the fractional Laplacian operator $(-\Delta)^{s}$ is defined by
\[(-\Delta)^{s}u(x)= \text{C.V.}\Big(\gamma_{n, s}\int_{\mathbb{R}^{n}}\frac{u(x)-u(y)}{|x-y|^{n+2s}}dy\Big)=
\lim\limits_{\epsilon\rightarrow 0}\gamma_{n, s}\int_{\mathbb{R}^{n}\backslash B_{\epsilon}(x)}\frac{u(x)-u(y)}{|x-y|^{n+2s}}dy,\]
where C.V. stands for the Cauchy principal value and
\[\gamma_{n,s}:=s2^{2s}\pi^{-\frac{n}{2}}\frac{\Gamma(\frac{n+2s}{2})}{\Gamma(1-s)}\]
is a normalization constant. Let
\[L_{s}(\mathbb{R}^{n})=\Big\{u\in L^{1}_{\text{loc}}(\mathbb{R}^{n})|\int_{\mathbb{R}^{n}}\frac{|u|}{1+|x|^{n+2s}}dx<\infty\Big\}.\]
It is well known that if  $u\in C^{2}(B_{R}(0))\cap L_{s}(\mathbb{R}^{n})$, then $(-\Delta)^{s}u$ is well defined.

Throughout the rest of discussions, we make use the following assumptions on the nonlinearity $W:\mathbb{R}\to [0, \infty)$:
\begin{itemize}
\item[(H1)] $W\in C^{2}(\mathbb{R}:[0, \infty))$.
\item[(H2)]$\{W=0\}=\{\pm1\}$ and $W''(\pm1)>0$.
\item[(H3)] There exist $p\in (1, \infty)$ and a constant $\mathbf{c}_{W}>0$ such that for all $t\in\mathbb{R}$,
\[\frac{1}{\mathbf{c}_{W}}(|t|^{p-1}-1)\leq |W'(t)|\leq \mathbf{c}_{W}(|t|^{p-1}+1).\]
\end{itemize}
Clearly, those set of assumptions are satisfied by the  prototypical double well potential $W(t)=(1-t^{2})^{2}/4$.

The Allen-Cahn energy functional associated to \eqref{fractional Allen Cahn equation} is
\begin{equation}\label{Euler Lagrange functional}
\mathcal{E}_{\epsilon}(u, B_{R}(0)):=\mathcal{E}(u, B_{R}(0))+\frac{1}{\epsilon^{2s}}\int_{B_{R}(0)}W(u)dx,
\end{equation}
where
\[\begin{aligned}
\mathcal{E}(u, B_{R}(0)):=&\frac{\gamma_{n, s}}{4}\iint_{B_{R}(0)\times B_{R}(0)}\frac{|u(x)-u(y)|^{2}}{|x-y|^{n+2s}}dxdy\\
&+\frac{\gamma_{n, s}}{2}\iint_{B_{R}(0)\times(\mathbb{R}^{n}\backslash B_{R}(0)}\frac{|u(x)-u(y)|^{2}}{|x-y|^{n+2s}}dxdy.
\end{aligned}\]
For a function $u$ in the Hilbert space
\[H^{s}(\mathbb{R}^{n}):=\Big\{u:\frac{\gamma_{n, s}}{2}\iint_{\mathbb{R}^{n}\times\mathbb{R}^{n}}\frac{|u(x)-u(y)|^{2}}{|x-y|^{n+2s}}dxdy+\int_{\mathbb{R}^{n}}u^{2}dx<\infty\Big\},\]
the energy functional \eqref{Euler Lagrange functional} is finite. In order to take the Dirichlet boundary condition \eqref{Dirichlet boundary condition} into account, one need to restrict to the class of functions given by the affine space $H_{g_{\epsilon}}^{s}(B_{R}(0)):=g_{\epsilon}+H^{s}_{00}(B_{R}(0))$ with
\[H^{s}_{00}(B_{R}(0)):=\{u:u\in H^{s}(\mathbb{R}^{n})~\text{and}~u=0~\text{almost everywhere in}~\mathbb{R}^{n}\backslash B_{R}(0)\}.\]
Note that $H^{s}_{00}(B_{R}(0))$ is also a Hilbert space. Using \eqref{Euler Lagrange functional}, we can define weak solutions of \eqref{fractional Allen Cahn equation} in the following way.
\begin{defi}[Weak solutions]\label{definition of weak solutions}
A function $u\in H_{g_{\epsilon}}^{s}(B_{R}(0))$ is a \textbf{\textit{weak solution}} of \eqref{fractional Allen Cahn equation} if $u$ is a  critical point of $\mathcal{E}_{\epsilon}(u, B_{R}(0))$ with respect to perturbations supported in $B_{R}(0)$.
\end{defi}
Among the class of weak solutions to \eqref{fractional Allen Cahn equation}, the energy minimizes of \eqref{Euler Lagrange functional} play a special role.  The link between asymptotic behavior of sequences of minimizes (see \cite[Definition 1.1]{Savin-Valdinoci2012}) to \eqref{Euler Lagrange functional} and the nonlocal minimal surfaces introduced by Caffarelli-Roquejoffre-Savin \cite{Caffarelli-Roquejoffre-Savin2010} is investigated by Savin and  Valdinoci \cite{Savin-Valdinoci2012} through a $\Gamma-$convergence analysis. Their main results show   that if $g_{\epsilon}\to g$ in $L^{1}_{loc}(\mathbb{R}^{n})$ for some function $g$ satisfying $|g|=1$ almost everywhere in $\mathbb{R}^{n}\backslash B_{R}(0)$, if   $u_{\epsilon}\in H_{g_{\epsilon}}^{s}(B_{R}(0))$ (for each $\epsilon\in(0, 1)$) is a energy minimizer of \eqref{Euler Lagrange functional} and
\begin{equation}\label{uniform energy bound}
\sup_{\epsilon\in (0,1)}\mathcal{E}_{\epsilon}(u_{\epsilon}, B_{R}(0))<\infty,
\end{equation}
then for any sequence $\epsilon_{j}\downarrow 0$, there exists a function $u_{\ast}\in H^{s}_{g}(B_{R}(0))$ such that (up to a subsequence) $u_{\epsilon_{j}}\to u_{\ast}$ strongly in $L^{1}_{loc}(\mathbb{R}^{n})$. Moreover, there exists a  a set $E_{\ast}\subset\mathbb{R}^{n}$ which is a minimizer of the  $2s-$perimeter such that $u_{\ast}=\chi_{E_{\ast}}-\chi_{\mathbb{R}^{n}\backslash E_{\ast}}$. Here for a  bounded smooth open set $\Omega\subset\mathbb{R}^{n}$, the $2s-$perimeter of a Borel set $E\subset\mathbb{R}^{n}$ relative to $\Omega$ is defined by
\begin{equation}\label{2s-perimeter}
\begin{aligned}
P_{2s}(E, \Omega):=&\int_{E\cap\Omega}\int_{E^{c}\cap \Omega}\frac{1}{|x-y|^{n+2s}}dxdy+\int_{E\cap \Omega}\int_{E^{c}\backslash\ \Omega}\frac{1}{|x-y|^{n+2s}}dxdy\\
&+\int_{E\backslash \Omega}\int_{E^{c}\cap\Omega}\frac{1}{|x-y|^{n+2s}}dxdy\end{aligned}
\end{equation}
with $E^{c}:=\mathbb{R}^{n}\backslash E$ being the complementary set of $E$. The minimality of $E_{\ast}$ implies the first variation of the $2s-$perimeter vanishes at $E_{\ast}$. This is equivalent to
\begin{equation}\label{First variation of 2s perimeter}
\delta P_{2s}\big(E_{\ast}, B_{R}(0)\big)[X]:=\Big [ \frac{d}{dt}P_{2s}\big(\phi_{t}(E_{\ast}), B_{R}(0)\big)\Big ]\Big|_{t=0}=0
\end{equation}
for any vector field $X\in C^{1}_{0}(B_{R}(0); \mathbb{R}^{n})$, where $\{\phi_{t}\}_{t\in\mathbb{R}}$ is the flow generated by $X$.

Later this convergence result was generalized by Millot, Sire and the first author  \cite{Millot-Sire-Wang2019}  to weak solutions (as defined in Definition \ref{definition of weak solutions}). Under the energy condition \eqref{uniform energy bound}, they proved that weak solutions $u_\epsilon$ still converges as in the above minimizing case, but now the limit is a    stationary nonlocal minimal surface, which is defined in the following way.
\begin{defi}[Stationary nonlocal minimal surfaces]
Let $\Omega\subset\mathbb{R}^{n}$ be a  bounded smooth open set. Let $E\subset\mathbb{R}^{n}$ be a Borel set satisfying $P_{2s}(E, \Omega)<\infty$ and
\[\delta P_{2s}(E, \Omega)[X]=0\]
for any vector field $X\in C^{1}_{0}(\Omega; \mathbb{R}^{n})$. Then $\partial E$ is referred as a \textbf{\textit{stationary nonlocal 2s-minimal surface}} in $\Omega$.
\end{defi}

\subsection{Main results} In this paper, we will follow the setting by Millot, Sire and the first author in \cite{Millot-Sire-Wang2019}, to investigate the regularity of stationary nonlocal minimal surfaces and the asymptotic behaviors of solutions to \eqref{fractional Allen Cahn equation}  when $\epsilon\to 0$. Our first result shows that the interface of a
phase transition behaves ``like a codimension one'' set.
\begin{thm}\label{thm volume estimate}
Let $\{g_{\epsilon}\}_{\epsilon\in (0, 1)}\subset C_{loc}^{0, 1}(\mathbb{R}^{n})$ be such that
\[\sup_{\epsilon\in(0, 1)}\|g_{\epsilon}\|_{L^{\infty}(\mathbb{R}^{n}\backslash B_{40}(0))}<\infty\]
and $g_{\epsilon}\to g$  strongly in $L_{loc}^{1}(\mathbb{R}^{n}\backslash B_{40}(0))$ for a function $g$ satisfying $|g|=1$ almost everywhere in $\mathbb{R}^{n}\backslash B_{40}(0)$.
For each $\epsilon\in (0, 1)$, let $u_{\epsilon}\in H^{s}_{g_{\epsilon}}(B_{40}(0))\cap L^{\infty}(\mathbb{R}^{n})$ be a weak solution of \eqref{fractional Allen Cahn equation} with $R=40$. Assume
\begin{equation}\label{uniform bounded assumption}
1\leq \sup_{\epsilon\in(0, 1)}\|u_{\epsilon}\|_{L^{\infty}(\mathbb{R}^{n})}=\Lambda_{0}<\infty
\end{equation}
and
\begin{equation}\label{uniform energy assumption}
\sup_{\epsilon\in(0, 1)}\mathcal{E}_{\epsilon}(u_{\epsilon}, B_{40}(0))=\Lambda_{1}<\infty.
\end{equation}
Given $0<\tau<1$, there exist  positive constants $\mathbf{k}_{\ast}, c$ depending only on $n, s, \tau, \Lambda_{0}$, $\Lambda_{1}, W$ such that for each $r\in (\mathbf{k}_{\ast}\epsilon, 1)$,
\begin{equation}\label{volume estimate}
\mathscr{L}^{n}\Big(\mathcal{J}_{r}\big(\{|u_{\epsilon}|<1-\tau\}\cap B_{1}(0)\big)\Big)\leq cr.
\end{equation}
Here $\mathcal{J}_{r}(A)$ represents the open tubular neighborhood of radius $r$ of a set $A$ and $\mathscr{L}^{n}$ is the Lebesgue measure on $\mathbb{R}^{n}$.
\end{thm}
\begin{rmk}
We emphasize that \eqref{uniform energy assumption} is a technical assumption. In the special case that $u_{\epsilon}$ is a finite Morse index \footnote{The  definition of finite Morse index solutions can be found in \cite[Definition 1.14]{Caselli-Simon-Serra2024}.} solution of the equation
\begin{equation}\label{Allen-Cahn equation}
(-\Delta)^{s}u_{\epsilon}=\frac{1}{\epsilon^{2s}}(u_{\epsilon}-u_{\epsilon}^{3}),\quad\text{in}~\mathbb{R}^{n},
\end{equation}
  Cabr\'{e}-Cinti-Serra \cite{Cabre-Cinti-Serra2021} (for stable solutions) and Caselli-Simon-Serra \cite{Caselli-Simon-Serra2024} (for finite Morse index solutions) proved that the energy estimate \eqref{uniform energy assumption}  holds naturally. There are also a lot of results concerning energy estimates to solutions of \eqref{Allen-Cahn equation}, see for instance \cite{Cabre-Cinti2014, Cui-Li2020}.

On the other hand, there are  solutions which do not satisfy this assumption. An example is the scalings of the one dimensional periodic solutions to \eqref{Allen-Cahn equation} as constructed in Du-Gui \cite{Du-Gui2020}.
\end{rmk}
Theorem \ref{thm volume estimate} improves \cite[Theorem 7.1]{Millot-Sire-Wang2019} to the optimal estimate. For minimizers of the energy functional \eqref{Euler Lagrange functional}, the density estimates have been obtained by Savin and Valdinoci \cite[Theorem 1.4]{Savin-Valdinoci2014}. For the classical Allen-Cahn energy,
such density estimates were proved in \cite{Caffarelli-Cordoba1995} for the local minimizer of the functional.
Compared with the  results in \cite{Savin-Valdinoci2014} and \cite{Caffarelli-Cordoba1995}, Theorem \ref{thm volume estimate} is stronger in the sense that it gives a uniform estimate in a neighborhood of the transition set.

As a direct application of Theorem \ref{thm volume estimate}, we can prove a sharp quantitative convergence rate for  the potential energy.
\begin{prop}\label{pro potential estimate}
Under the same assumptions of Theorem \ref{thm volume estimate}, we have
\begin{equation}
\int_{B_{1}(0)} W(u_{\epsilon}) dx\leq \left\{\begin{array}{lll}c\epsilon^{4s},&\quad\text{if}~~0<s<\frac{1}{4},\\
c\epsilon|\log\epsilon|,&\quad\text{if}~~s=\frac{1}{4},\\
c\epsilon,&\quad\text{if}~~\frac{1}{4}<s<\frac{1}{2},
\end{array}
\right.\end{equation}
where $c$ is a positive constant depending only on $n, s, \Lambda_{0} , \Lambda_{1}, W$.
\end{prop}
\begin{rmk}
Let $w$ be the one-dimensional layer solution of \eqref{Allen-Cahn equation}.
For each $\epsilon\in (0, 1)$, we set $u_{\epsilon}(x)=w(\epsilon^{-1} (a\cdot x +b))$, then it follows from \cite[Theorem 2.7]{Caber-Sire2015} that $u_{\epsilon}$ satisfies the potential estimate in Proposition \ref{pro potential estimate}. Hence this potential estimate is sharp.
\end{rmk}
\begin{rmk}
Proposition \ref{pro potential estimate} is a refinement of \cite[Corollary 7.5]{Millot-Sire-Wang2019}.  In \cite[Proposition 6.2]{Cabre-Cinti-Serra2021}, Cabre-Cinti-Serra obtained  a similar potential estimate for stable solutions of \eqref{Allen-Cahn equation}. They proved that
\begin{equation}
\int_{B_{1}(0)} W(u_\epsilon) dx\leq c\epsilon^{\min\{4s, \frac{1+2s}{2}\}}.
\end{equation}
\end{rmk}
In the same spirit as Theorem \ref{thm volume estimate}, we can also prove an optimal perimeter estimate for the stationary nonlocal minimal surfaces.
\begin{thm}\label{the perimeter estimate}
Let $E_{\ast}\subset\mathbb{R}^{n}$ be a Borel set satisfying \[P_{2s}(E_{\ast}, B_{R}(0))=\Lambda_{0}<\infty\] and \eqref{First variation of 2s perimeter}. Then for any $R_{1}<R$, $\partial{E_{\ast}}$ is $(n-1)$-rectifiable (see \cite{simon1984lectures}) in $B_{R_{1}}(0)$. Moreover, there exists a positive constant $c$ depending on $n, s, R, R_{1}, \Lambda_{0}$ such that
\[\mathcal{H}^{n-1}(\partial{E_{\ast}}\cap B_{R_{1}}(0))\leq cR_{1}^{n-1},\]
where $\mathcal{H}^{n-1}$ is the $(n-1)$-dimensional Hausdorff measure.
\end{thm}
In particular, Theorem \ref{the perimeter estimate}  implies that stationary nonlocal minimal surfaces enjoy a quite strong regularity. This kind of phenomenon was first discovered by Cinti-Serra-Valdinoci \cite[Theorem 1.1]{Cinti-Serra-Valdinoci2019} when $E_{\ast}$ is a stable set (see \cite[Definition 1.6]{Cinti-Serra-Valdinoci2019})   for the nonlocal perimeter functional. In this case, they proved that $\chi_{E_{\ast}}$ is a function of bounded variation (see \cite[Definition 1.1]{Giusti1984})  with a universal perimeter estimate. (Then its reduced boundary is rectifiable by De Giorgi's rectifiability theorem.) Their proofs make crucial use of the stability property.  Instead, we can obtain these results using only the monotonicity formula and the compactness for stationary nonlocal minimal surfaces established in \cite[Theorem 6.7]{Millot-Sire-Wang2019}.

\subsection{The Dirichlet-to-Neumann Operator}

To consider the nonlocal problem, we will realize it as a local problem in $\mathbf{R}^{n+1}_{+}$ with a nonlinear Neumann boundary condition. More precisely, let $\mathcal{K}_{n,s}:\mathbb{R}_{+}^{n+1}\to[0, \infty)$
be defined by
\[\mathcal{K}_{n, s}(x,z):=\sigma_{n, s}\frac{z^{2s}}{|\mathbf{x}|^{n+2s}},\]
where $\mathbf{x}=(x, z)\in\mathbb{R}_{+}^{n+1}=\mathbb{R}^{n}\times (0, +\infty)$ and $\sigma_{n, s}$ is a positive constant such that for every $z>0$,
\[\int_{\mathbb{R}^{n}}\mathcal{K}_{n, s}(x, z)dx=1.\]
For a  function $u\in L_{s}(\mathbb{R}^{n})\cap L^{\infty}(\mathbb{R}^{n})$, we denote by $U$ its extension to $\mathbb{R}_{+}^{n+1}$ given by the convolution
\begin{equation}\label{fractional extension}
U(x, z)=\sigma_{n, s}\int_{\mathbb{R}^{n}}\frac{z^{2s}}{(|x-y|^{2}+z^{2})^{\frac{n+2s}{2}}}u(y)dy.
\end{equation}
Caffarelli and Silvestre \cite{Caffarelli-Silvestre2007} proved that $U$ satisfies the equation
\begin{equation}\label{extension equation}
 \left\{\begin{array}{lll}
{\rm div}(z^{a}\nabla U)=0,\qquad~~ &\text{in }\mathbb{R}^{n+1},\\
U(x, 0)=u,\qquad~~ &\text{on }\partial\mathbb{R}^{n+1},\\
-d_{s}\lim_{z\rightarrow 0}z^{a}\partial_{z}U(x, z)=(-\Delta)^{s}u,\qquad~~ &\text{on }\partial\mathbb{R}^{n+1},
\end{array}
\right.
\end{equation}
where $a=1-2s$ and $d_{s}$ is a normalization constant depending only on $s$.

We denote $\mathcal{B}_{R}(\mathbf{x})$ as the ball in $\mathbb{R}^{n+1}$ with radius $R$
and center $\mathbf{x}$, $\mathcal{B}_{R}^{+}(\mathbf{x})$ as $\mathcal{B}_{R}(\mathbf{x})\cap \mathbb{R}_{+}^{n+1}$. We also set $\partial'\mathcal{B}_{R}^{+}(\mathbf{x})=\partial \mathcal{B}_{R}^{+}(\mathbf{x})\cap \mathbb{R}_{+}^{n+1}$ and $\partial^{0}\mathcal{B}_{R}^{+}(\mathbf{x})=\partial \mathcal{B}_{R}^{+}(\mathbf{x})\cap \partial\mathbb{R}_{+}^{n+1}$. If there is no ambiguity, we will identity $\partial\mathbb{R}_{+}^{n+1}$ with $\mathbb{R}^{n}$, so $\partial^{0}\mathcal{B}_{R}^{+}(\mathbf{x})=
B_{R}(x)$. Let $\mathbf{0}$ be the origin in $\mathbb{R}^{n+1}$ and
\[L^{2}(\mathcal{B}_{R}^{+}(\mathbf{0}), z^{a}d\mathbf{x}):=\{U\in L_{loc}^{1}(\mathcal{B}_{R}(\mathbf{0})):|z|^{\frac{a}{2}}U\in L^{2}(\mathcal{B}_{R}(\mathbf{0}))\}\]
be the weighed $L^{2}-$ space normed by
\[\|U\|_{L^{2}(\mathcal{B}_{R}^{+}(\mathbf{0}), z^{a}d\mathbf{x})}^{2}:=\int_{\mathcal{B}_{R}^{+}(\mathbf{0})}z^{a}U^{2}d\mathbf{x}.\]
We also introduce the weighted Sobolev space
\[H^{1}(\mathcal{B}_{R}^{+}(\mathbf{0}), z^{a}d\mathbf{x}):=\{U\in L^{2}(\mathcal{B}_{R}^{+}(\mathbf{0}), z^{a}d\mathbf{z}):|\nabla U|\in L^{2}(\mathcal{B}_{R}^{+}(\mathbf{0}), z^{a}d\mathbf{z})\}\]
normed by
\[\|U\|_{H^{1}(\mathcal{B}_{R}^{+}(\mathbf{0}), z^{a}d\mathbf{x})}^{2}:=\int_{\mathcal{B}_{R}^{+}(\mathbf{0})}z^{a}U^{2}d\mathbf{x}+\int_{\mathcal{B}_{R}^{+}(\mathbf{0})}z^{a}|\nabla U|^{2}d\mathbf{x}.\]
For a function $U\in H^{1}(\mathcal{B}_{R}^{+}(\mathbf{0}), z^{a}d\mathbf{x})\cap L^{\infty}(B_{R}(0))$, the Dirichlet energy
\begin{equation}\label{H1 Energy quantity}
E(U, \mathcal{B}_{R}^{+}(\mathbf{0})):=\frac{d_{s}}{2}\int_{\mathcal{B}_{R}^{+}(\mathbf{0})}z^{a}|\nabla U|^{2}d\mathbf{x}
\end{equation}
is well defined. If $u_{\epsilon}\in H^{s}_{g_{\epsilon}}(B_{R}(0))\cap L^{\infty}(\mathbb{R}^{n})$ is a weak solution of \eqref{fractional Allen Cahn equation} and $U_{\epsilon}$ is the extension of $u_{\epsilon}$ on $\mathbb{R}^{n+1}_{+}$. Then
 \begin{equation}
d_{s}\int_{\mathcal{B}_{R}^{+}(\mathbf{0})}z^{a}\nabla U_{\epsilon}\cdot\nabla\phi d\mathbf{x}+\frac{1}{\epsilon^{2s}}\int_{B_{R}(0)}W'(u_{\epsilon})\phi dx=0
\end{equation}
for every $\phi\in H^{1}(\mathcal{B}_{R}^{+}(\mathbf{0}), z^{a}d\mathbf{x})\cap L^{p}(B_{R}(0))$ compactly supported in $\mathcal{B}_{R}^{+}(\mathbf{0})\cup B_{R}(0)$, where $1\leq p<\infty$. This implies $U_{\epsilon}$ is a critical point of the weighed Dirichlet energy
\begin{equation}\label{Energy quantity}
\textbf{E}_{\epsilon}(U, \mathcal{B}^{+}_{R}(\mathbf{x}_{0}))=\textbf{E}(U, \mathcal{B}^{+}_{R}(\mathbf{0}))+\frac{1}{\epsilon^{2s}}\int_{B_{R}(x_{0})}W(U(x, 0))dx.
\end{equation}
In other words, $U_{\epsilon}$ is a \textbf{\textit{weak solution}} of
\begin{equation}\label{Extension equation}
\left\{\begin{array}{lll}
\textrm{div}(z^{a}\nabla U_{\epsilon})=0, &\quad~\text{in}~\mathcal{B}_{R}^{+}(\mathbf{0}),\\
d_{s}\lim_{z\to 0}z^{a}\partial_{z}U_{\epsilon}=\frac{1}{\epsilon^{2s}}W'(u_{\epsilon}),&\quad~\text{on}~B_{R}(0).
\end{array}
\right.
\end{equation}

In the case that $E_{\ast}$ is a stationary nonlocal $2s-$minimal surface in $B_{R}(0)$, the extension of $u_{\ast}=\chi_{E_{\ast}}-\chi_{\mathbb{R}^{n}\backslash E_{\ast}}$ on $\mathbb{R}^{n+1}_{+}$ satisfies
\begin{equation}\label{eqn1}
\left\{\begin{array}{lll}
\textrm{div}(z^{a}\nabla U_{\ast})=0,&\quad\text{in}~\mathbb{R}^{n+1}_{+},\\
|U_{\ast}|\leq 1,&\quad\text{in}~\mathbb{R}^{n+1}_{+},\\
|U_{\ast}|=1,&\quad\text{on}~\partial\mathbb{R}^{n+1}_{+}
\end{array}
\right.
\end{equation}
In addition to \eqref{eqn1}, we know from \cite[Remark 2.15]{Millot-Sire-Wang2019} and \cite[Formula (6.2)]{Millot-Sire-Wang2019} that
\begin{equation}\label{eqn2}
\delta\textbf{E}\big(U_{\ast}, \mathcal{B}^{+}_{R}(\mathbf{0})\cup\partial^{0}\mathcal{B}^{+}_{R}(\mathbf{0})\big)[\mathbf{X}]:=\Big [\frac{d}{dt}\textbf{E}(U_{\ast}\circ\Phi_{-t}, \mathcal{B}^{+}_{R}(\mathbf{0}))\Big]\Big|_{t=0}=0
\end{equation}
for each  vector field  $\mathbf{X}=(X, X_{n+1})\in C^{1}(\overline{\mathcal{B}^{+}_{R}(\mathbf{0})}; \mathbb{R}^{n+1})$ compactly supported in $\mathcal{B}^{+}_{R}(\mathbf{0})\cup\partial^{0}\mathcal{B}^{+}_{R}(\mathbf{0})$ and satisfying $X_{n+1}=0$ on $\partial^{0}\mathcal{B}^{+}_{R}(\mathbf{0})$,
where $\{\Phi_{t}\}_{t\in\mathbb{R}}$ denotes the flow on $\mathbb{R}^{n+1}$ generated by $\mathbf{X}$.

Using the notations in this subsection,  Theorem \ref{thm volume estimate} will follow from the following covering result.
\begin{thm}\label{themain1}
 For each $\epsilon\in (0, 1)$, let $U_{\epsilon}\in H^{1}(\mathcal{B}_{40}^{+}(\mathbf{0}), z^{a}d\mathbf{x})\cap L^{\infty}(\mathcal{B}_{40}^{+}(\mathbf{0}))$ be a weak solution of \eqref{extension equation}. Assume
\begin{equation}\label{uniform bound assumption'}
1\leq \sup_{\epsilon\in(0, 1)}\|U_{\epsilon}\|_{L^{\infty}(\mathcal{B}_{40}^{+}(\mathbf{0}))}\leq \Lambda_{0}<\infty
\end{equation}
and
\begin{equation}\label{uniform energy assumption'}
\sup_{\epsilon\in(0, 1)}\textbf{E}_{\epsilon}(U_{\epsilon}, \mathcal{B}_{40}^{+}(\mathbf{0}))=\Lambda_{1}<\infty.
\end{equation}
Given $0<\tau<1$, there exist  positive constants $\mathbf{k}, c$ depending only on $n, s, \tau,\Lambda_{0}$, $\Lambda_{1}, W$ such that the following holds.

If $\mathbf{k}\epsilon<1$ and $r\in (\mathbf{k}\epsilon, 1)$, we can find a collection of balls $\{B_{r}(x_{i})\}_{i}$ which satisfy
\[\big\{|U_{\epsilon}(x, 0)|\leq 1-\tau\big\}\cap B_{1}(0)\subset\mathop{\cup}_{i}B_{r}(x_{i})\]
and
\[\#\{x_{i}\}_{i}\leq c(n , s, \tau, \Lambda_{0}, \Lambda_{1}, W)r^{1-n},\]
where $\#\{x_{i}\}_{i}$ is the number of the points $x_{i}$. In particular,  we have
\[\mathscr{L}^{n}\Big(\mathcal{J}_{r}\big(\{|u_{\epsilon}|\leq 1-\tau\}\cap B_{1}(0)\big)\Big)\leq c(n ,s, \tau, \Lambda_{0}, \Lambda_{1}, W)r.\]
\end{thm}
\begin{rmk}
In Theorem \ref{themain1}, the assumption $\mathbf{k}\epsilon<r$ is essential.  But the restriction that $r<1$ can be removed. Indeed, if $r\geq 1$, we can always find a collection of balls satisfying Theorem \ref{themain1}, hence Theorem \ref{themain1} is trivial.
\end{rmk}
In order to establish Theorem \ref{themain1}, we will adopt the quantitative stratification technique developed by Naber-Valtorta in  \cite{Naber-Valtorta2017}.  This technique is  extremely useful and has found many  applications in recent years, see for instance \cite{Edelen-Engelstein2019, Sire2025, DeLellis2018, WangYu, Hirsch, Fu-Wang-Zhang2024, Wang-Zhang2024}.

 However, arguments applied to the study of harmonic maps in \cite{Naber-Valtorta2017} cannot be directly used here due
to several inherent difficulties in this model.
Indeed, if $g_{\epsilon}$ and $u_{\epsilon}$ satisfies the assumptions in Theorem \ref{thm volume estimate}, then $u_{\epsilon}\in C^{\alpha}(B_{R}(0))$ for some $\alpha\in (0, 1)$. Therefore, we can not define singular set and strata as in \cite{Naber-Valtorta2017}. In our proof, the transition set $\{|u_{\epsilon}|\leq 1-\tau\}$ will play the role of singular set and the parameter $\epsilon$ will be served as the threshold scale.

This paper will be organized as follows. In Section 2 and Section 3, we collect some preliminary results. In Section 4, we recall two Reifenberg type theorems which will be used as black boxes. In Section 5, we establish the relation between Reifenberg-type theorems and the monotonicity formula. In Section 6, we prove that if we have a collection of disjoint balls with small energy drop at the center, then we can use the Reifenberg-type theorems to obtain packing estimates. In Section 7, we prove a crucial dichotomy. In Section 8, we combine the results obtained in the previous sections to construct a corona-type decomposition. In Section 9 and Section 10, we give the proof of the main results.

\textbf{Notations.}
\begin{itemize}
\item If $\mu$ is a locally finite Borel measure on $\mathbb{R}^{n}$ and $S$ is a Borel set, we use $\mu \llcorner S$ to denote the restriction of $\mu$ on $S$.
\item If $\mu$ is a  Borel measure on $\mathbb{R}^{n}$, then $\textrm {spt}\mu$ is the support of $\mu$.
\end{itemize}

\section{Preliminaries}
\setcounter{equation}{0}
Given  $\mathbf{x}_{0}=(x_{0}, 0)\in\partial^{0}\mathcal{B}^{+}_{R}(\mathbf{0})$ and $r>0$ such that $\mathcal{B}_{r}^{+}(\mathbf{x}_{0})\subset \mathcal{B}^{+}_{R}(0)$. For each $U\in H^{1}(\mathcal{B}_{R}^{+}(\mathbf{0}), z^{a}d\mathbf{x})$, we set
\begin{equation}\label{Monotone quantity}
\mathbf{\Theta}_{U}^{\epsilon}(r, x_{0})=\frac{1}{r^{n-2s}}\textbf{E}_{\epsilon}(U, \mathcal{B}^{+}_{r}(\mathbf{x}_{0})),~~ \mathbf{\Theta}_{U}(r, x_{0}):=\frac{1}{r^{n-2s}}\textbf{E}(U, \mathcal{B}_{r}^{+}(\mathbf{x}_{0}))
\end{equation}

We begin with the following fundamental monotonicity formula, see \cite[Lemma 4.2]{Millot-Sire-Wang2019}. (Other results in this section can also be found in this paper.)
\begin{lem}\label{lem Allen-Cahn monotonicity formula}
If $U_{\epsilon}\in H^{1}(\mathcal{B}^{+}_{R}(\mathbf{0}), z^{a}d\mathbf{x})\cap L^{\infty}(\mathcal{B}^{+}_{R}(\mathbf{0}))$ is a weak solution of \eqref{extension equation}, then for any $\rho<r$,
\begin{equation}\label{Allen-Cahn monotonicity formula}
\begin{aligned}
\mathbf{\Theta}_{U_{\epsilon}}^{\epsilon}(r, x_{0})-\mathbf{\Theta}_{U_{\epsilon}}^{\epsilon}(\rho, x_{0})=&d_{s}\int_{\mathcal{B}_{r}^{+}(\mathbf{x}_{0})\backslash \mathcal{B}_{\rho}^{+}(\mathbf{x}_{0})}z^{a}\frac{|(\mathbf{x}-\mathbf{x}_{0})\cdot\nabla U_{\epsilon}|^{2}}{|\mathbf{x}-\mathbf{x}_{0}|^{n+2-2s}}d\mathbf{x}\\
&+\frac{2s}{\epsilon^{2s}}r^{-n+2s-1}\int_{B_{r}(x_{0}\backslash B_{\rho}(x_{0})}W(u_{\epsilon})dx.
\end{aligned}
\end{equation}
In particular, $\mathbf{\Theta}_{U_{\epsilon}}^{\epsilon}(r, x_{0})$ is  increasing in $r$.
\end{lem}
The next theorem deals with the compactness  for weak solutions to \eqref{extension equation}.
\begin{thm}[Compactness]\label{convergence}
    For a given sequence $\epsilon_{j}\downarrow 0$. For each $j\in\mathbb{N}$, let $U_{\epsilon_{j}}\in H^{1}(\mathcal{B}^{+}_{R}(\mathbf{0}), z^{a}d\mathbf{x})\cap L^{\infty}(\mathcal{B}^{+}_{R}(\mathbf{0}))$ be a weak solution of \eqref{extension equation}. If
    \[\sup_{j}\textbf{E}_{\epsilon_{j}}(U_{\epsilon_{j}}, \mathcal{B}^{+}_{R}(\mathbf{x}_{0}))\leq \Lambda,\]
where $\Lambda$ is a fixed positive constant. Then there exist a (not relabeled) subsequence, a function $U_{\ast}\in H^{1}(\mathcal{B}_{R}^{+}(\mathbf{0}), z^{a}d\mathbf{x})$ and an open set $E_{\ast}\subset\mathbb{R}^{n}$ of finite $2s-$ perimeter in $B_{R}(0)$ such that the following holds.
\begin{itemize}
\item[(i)] $U_{\ast}(x, 0)=\chi_{E_{\ast}}-\chi_{\mathbb{R}^{n}\backslash E_{\ast}}$ in $B_{R}(0)$ and $U_{\ast}$ is a solution of \eqref{eqn1} satisfying \eqref{eqn2},
\item[(ii)] $U_{\epsilon_{j}}\to U_{\ast}$
strongly in $H^{1}(\mathcal{B}_{R}^{+}(\mathbf{0})\cup B_{R}(0), z^{a}d\mathbf{x})$ as $j\to\infty$,
\item[(iii)] $U_{\epsilon_{j}}(x, 0)\to \chi_{E_{\ast}}-\chi_{\mathbb{R}^{n}\backslash E_{\ast}} $ in $C^{0}_{loc}(B_{R}(0)\backslash\partial E_{\ast})$,
\item[(iv)]$\epsilon_{j}^{-2s}W(U_{\epsilon_{j}}(x, 0))\to 0$ strongly in $L^{1}(B_{R}(0))$.
\end{itemize}
\end{thm}
The compactness for sequences of weak solutions of \eqref{extension equation} implies the continuity of the quantity defined in \eqref{Monotone quantity}. In particular, if  $x_{j}\to x$ and $r_{j}\to r>0$ as $j\to\infty$, then (up to a subsequence)
\[\mathbf{\Theta}_{U_{\epsilon_{j}}}^{\epsilon_{j}}(r_{j}, x_{j})\to\mathbf{\Theta}_{U_{\ast}}(r, x), \quad\text{as}~j\to\infty.\]

We also need the following monotonicity formula which holds for solutions of \eqref{eqn1} satisfying \eqref{eqn2}.
\begin{lem}[Monotonicity formula]\label{lem Monotonicityformula} Let $U$ be a solution of \eqref{eqn1} satisfying \eqref{eqn2}. Then for every $\rho<r$,
\begin{equation}\label{Monotonicityformula}
\mathbf{\Theta}_{U}(r, x_{0})-\mathbf{\Theta}_{U}(\rho, x_{0})=d_{s}\int_{\mathcal{B}_{r}^{+}(\mathbf{x}_{0})\backslash \mathcal{B}_{\rho}^{+}(\mathbf{x}_{0})}z^{a}\frac{|(\mathbf{x}-\mathbf{x}_{0})\cdot\nabla U|^{2}}{|\mathbf{x}-\mathbf{x}_{0}|^{n+2-2s}}d\mathbf{x}.
\end{equation}
\end{lem}
Lemma \ref{lem Monotonicityformula} implies the density function
\begin{equation}\label{densityfunction}
\mathbf{\Theta}_{U}(x_{0})=\lim_{r\downarrow0}\mathbf{\Theta}_{U}(r, x_{0})=\lim_{r\downarrow0}\frac{1}{r^{n-2s}}\textbf{E}(U,\mathcal{B}_{r}^{+}(\mathbf{x}_{0}))
\end{equation}
is well defined. The function $\mathbf{\Theta}_{U}$  is upper semi-continuous in the following sense.
\begin{lem}[Upper semi-continuity]\label{Uppersemicontinuity}
Let $\{U_{j}\}$ be a sequence of  solutions of \eqref{eqn1} satisfying \eqref{eqn2}. Assume  $U_{j}\to U_{\infty}$ weakly in $H^{1}(\mathcal{B}^{+}_{R}(\mathbf{0}), z^{a}d\mathbf{x})$ and $U_{j}\to U_{\infty}$ strongly in $H^{1}_{loc}(\mathcal{B}^{+}_{R}(\mathbf{0})\cup\partial^{0}\mathcal{B}^{+}_{R}(\mathbf{0}), z^{a}d\mathbf{x})$. If $\{x_{j}\}\subset B_{R}(0)$ is a sequence converging to $x\in B_{R}(0)$, then
\[\lim_{j\to\infty}\mathbf{\Theta}_{U_{j}}(x_{j})\leq \mathbf{\Theta}_{U_{\infty}}(x).\]
\end{lem}
In the end of this section, we state the following clearing-out property  which can be seen as a small-energy regularity result.
\begin{lem}[Clear out]\label{Clear out}
There exist positive constants $\eta_{b}, R_{b}$ such that the following holds.

Given $\epsilon_{j}\downarrow 0$. For each $j$, let $U_{\epsilon_{j}}\in H^{1}(\mathcal{B}^{+}_{R}(\mathbf{0}), z^{a}d\mathbf{x})\cap L^{\infty}(\mathcal{B}^{+}_{R}(\mathbf{0}))$ be a weak solution of \eqref{Extension equation} such that $\|U_{\epsilon_{j}}\|_{L^{\infty}(\mathcal{B}^{+}_{R}(\mathbf{0}))}\leq b$. If $\rho\leq R_{b}$ and
\[\liminf_{j\to\infty}\mathbf{\Theta}^{\epsilon_{j}}_{U_{\epsilon_{j}}}(\rho, \mathbf{0})\leq \eta_{b}.\]
Then (up to a subsequence)  either $U_{\epsilon_{j}}\to 1$ or $U_{\epsilon_{j}}\to-1$ as $j\to\infty$ uniformly on $B_{\rho/4}(\mathbf{0})$ .
\end{lem}
\section{Symmetry of functions}
\setcounter{equation}{0}
\begin{defi}
For any $x_{0}\in\mathbb{R}^{n}$, a function $U$ defined on $\mathbb{R}_{+}^{n+1}$  is homogeneous about the point $\mathbf{x}_{0}=(x_{0}, 0)$ if for any $\lambda>0$,
\[U(\mathbf{x}_0+\lambda\mathbf{x})=U(\mathbf{x}_0+\mathbf{x}).\]
For each $k\in\{0, 1, \cdots n\}$, we say a  function $U$ is $k$-symmetric if $U$ is homogeneous about $\mathbf{x}_{0}$ and there is a $k$-dimensional plane $L^{k}\subset\mathbb{R}^{n}$ such that
\[U(\mathbf{x}+(v, 0))=U(\mathbf{x}),\quad\forall v\in L^{k}.\]
In particular, if $U$ is $0$-symmetric,  then $U$ is homogeneous about $\mathbf{x}_{0}$.
\end{defi}
\begin{defi}
Given $r>0, x_{0}\in\mathbb{R}^{n}$ and $k\in\{0, 1, \cdots n\}$, we say a function $U\in L^{2}(\mathcal{B}_{r}^{+}(\mathbf{x}_{0}), z^{a}d\mathbf{x})$ is $(k, \epsilon)-$symmetric about  $\mathbf{x}_{0}=(x_{0}, 0)$  in $\mathcal{B}_{r}^{+}(\mathbf{x}_{0})$ if
\[r^{2s-2-n}\int_{\mathcal{B}_{r}^{+}(\mathbf{x}_{0})}z^{a}|U-\bar{U}|^{2}d\mathbf{x}<\epsilon\]
for some $k-$symmetric $\bar{U}\in L^{2}(\mathcal{B}_{r}^{+}(\mathbf{x}_{0}), z^{a}d\mathbf{x})$.
\end{defi}
\begin{defi}[Nonlocal stationary cone]
A function $\phi$ is a nonlocal stationary cone if $\phi$ is homogeneous with respect to $\mathbf{0}$
and $\phi\in H^{1}(\mathcal{B}_{1}^{+}(\mathbf{0}), z^{a}d\mathbf{x})\cap L^{\infty}(\mathcal{B}_{1}^{+}(\mathbf{0}))$ is a solution of \eqref{eqn1} satisfying \eqref{eqn2}.
\end{defi}
\begin{rmk}\label{n-symmetric nonlocal minimal cone}
If $\phi$ is a $n-$symmetric nonlocal stationary cone, then either $\phi=1$ or $\phi=-1$.
\end{rmk}
The next lemma is the key step to apply Federer’s dimension reduction principle. The proof can be found in \cite[Lemma 6.18]{Millot-Sire-Wang2019}.
\begin{lem}\label{Dimension reduction principle}
Let $\phi$ be a nonlocal stationary cone. Then
\[\mathbf{\Theta}_{\phi}(y)\leq \mathbf{\Theta}_{\phi}(0),\quad\text{for any}~y\in\mathbb{R}^{n}.\]
In addition, the set
\[S(\phi):=\{y\in\mathbb{R}^{n}:\mathbf{\Theta}_{\phi}(y)=\mathbf{\Theta}_{\phi}(0)\}\]
is a linear subspace of  $\mathbb{R}^{n}$, and $\phi(\mathbf{x}+(y, 0))=\phi(\mathbf{x})$ for every $y\in S(\phi)$ and $\mathbf{x}\in\mathbb{R}_{+}^{n+1}$.
\end{lem}
Finally, we provide a useful lemma, which is an easy consequence of the uniqueness of the strong convergence limit.
\begin{lem}\label{lem symmetry of limitation}
Given a sequence $\epsilon_{j}\downarrow 0$, suppose for each $j\in\mathbb{N}$,   $U_{\epsilon_{j}}$ is a weak solution of \eqref{extension equation}.
Assume
\[\sup_{j}\textbf{E}_{\epsilon_{j}}(U_{\epsilon_{j}}, \mathcal{B}_{R}^{+}(\mathbf{0}))<\infty\] and there exists a sequence
$\delta_{j}\downarrow 0$ such that for each $j$,
\[U_{\epsilon_{j}}~\text{is}~(0, \delta_{j})~\text{symmeteric in}~ \mathcal{B}^{+}_{R}(\mathbf{0}).\]
If $U_{\epsilon_{j}}\to U_{\ast}$ strongly in $H^{1}(\mathcal{B}_{R}^{+}(\mathbf{0})\cup B_{R}(0), z^{a}d\mathbf{x})$ as $j\to\infty$, then $U_{\ast}$ is homogeneous in $\mathcal{B}^{+}_{R}(\mathbf{0})$.
\end{lem}
\begin{proof}
For each $j$, let $\bar{U}_{j}$ be a function which is homogeneous with respect to $\mathbf{0}$ such that
\[R^{2s-2-n}\int_{\mathcal{B}_{R}^{+}(\mathbf{0})}z^{a}|U_{\epsilon_{j}}-\bar{U}_{j}|^{2}d\mathbf{x}<\delta_{j}.\]
Then
\[\int_{\mathcal{B}_{R}^{+}(\mathbf{0})}z^{a}|U_{\ast}-\bar{U}_{j}|^{2}d\mathbf{x}\leq 2r^{n+2-2s}\delta_{j}+2\int_{\mathcal{B}_{R}^{+}(\mathbf{0})}z^{a}|U_{\epsilon_{j}}-\bar{U}_{j}|^{2}d\mathbf{x}\]
Since $U_{\epsilon_{j}}\to U_{\ast}$ strongly, it follows that $\bar{U}_{j}\to U_{\ast}$ strongly as $j\to\infty$. Since each $\bar{U}_{j}$ is homogeneous with respect to $\mathbf{x}_{0}$, $U_{\ast}$ is homogeneous with respect to $\mathbf{x}_{0}$.
\end{proof}

\section{Reifenberg type theorems}
\setcounter{equation}{0}
  Given $k\in\{1, 2, \cdots, n-1\}$ and $k+1$ points $y_{0}, \cdots, y_{k}\in\mathbb{R}^{n}$,  let $\langle y_{0}, \cdots, y_{k}\rangle$ be the affine $k-$space spanned by $y_{0}, \cdots, y_{k}$. For each locally finite Borel measure $\mu$ on $\mathbb{R}^{n}$, the Johns' $\beta_{\mu, 2}^{k}-$number is defined by
\begin{equation}\label{NOT1}
\beta_{\mu, 2}^{k}(x, r)^{2}=\inf\limits_{L^{k}}r^{-k-2}\int_{B_{r}(x)}d(y, L^{k})^{2}d\mu(y),
\end{equation}
where $d(y, L^{k})$ is the Euclidean distance from $y$ to $L^{k}$ and the infimum is over all affine $k-$plane. In \cite{Naber-Valtorta2017}, Naber and Valtorta proved the following two  Reifenberg type theorems, which are used to prove packing bounds and rectifiability. In this paper, we will use them as black boxes.
\begin{thm}[Discrete-Reifenberg \cite{Naber-Valtorta2017}]\label{theRec1}
Let $\{B_{r_{q}}(q)\}_{q}$ be a collection of disjoint balls, with $q\in B_{1}(0)$ and $0<r_{q}\leq 1$, and let $\mu$ be the packing measure
\[\mu=\sum_{q}r_{q}^{k}\delta_{q},\]
 where $\delta_{q}$ is the Dirac measure at $q$. There exists a positive constant $\delta_{DR}(n)$ such that if
\[\int_{0}^{2r}\int_{B_{r}(x)}\frac{1}{s}\beta_{\mu, 2}^{k}(y, s)^{2}d\mu(y)ds\leq \delta_{DR}(n)r^{k}, \quad\forall x\in B_{1}(0),~0<r\leq 1,\]
then
\[\mu(B_{1}(0))\equiv\sum\limits_{q}r_{q}^{k}\leq C_{DR}(n).\]
\end{thm}
\begin{thm}[Rectifiable-Reifenberg \cite{Naber-Valtorta2017}]\label{theRec2}
There exists a constant $\delta_{DR}(n)>0$ depending only on dimension such that, if a set $S$ satisfies
\[\int_{0}^{2r}\int_{B_{r}(x)}\frac{1}{s}\beta_{\mathcal{H}^{k}\llcorner S,2}^{k}(y, s)^{2}d(\mathcal{H}^{k}\llcorner S)(y) ds\leq \delta_{DR}(n)r^{k}, \quad\forall x\in B_{1}(0),~0<r\leq 1,\]
then $S\cap B_{1}(0)$ is $k-$rectifiable, and satisfies
\[\mathcal{H}^{k}(S\cap B_{1}(0)\cap B_{r}(x))\leq C_{DR}(n)r^{k}, \quad\forall x\in B_{1}(0),~0<r\leq 1,\]
where $\mathcal{H}^{k}$ is the $k-$dimensional Hausdorff measure.
\end{thm}
Finally, we present a version of above theorems that is more discrete in nature.
\begin{coro}\label{correc1}
Under the assumptions of Theorem \ref{theRec1}, if instead for every integer $i$,
\[\sum\limits_{j\leq 1-i}\int_{B_{2^{-i}}(x)}\beta_{\mu, 2}^{k}(y, 2^{-j})^{2}d\mu(y)\leq 2^{-2-2n}\delta_{DR}(n)2^{-ik},\]
then
\[\mu(B_{1}(0))\equiv\sum\limits_{q}r_{q}^{k}\leq C_{DR}(n).\]
\end{coro}

\begin{coro}\label{correc2}
Under the assumptions of Theorem \ref{theRec1}, if instead $S$ satisfies
 for every integer $i$,
\[\sum\limits_{j\leq 1-i}\int_{B_{2^{-i}}(x)}\beta_{\mathcal{H}^{k}\llcorner S, 2}^{k}(y, 2^{-j})^{2}d\mu(y)\leq 2^{-2-2n}\delta_{DR}(n)2^{-ik},\]
then $S\cap B_{1}(0)$ is $k-$rectifiable, and satisfies
\[\mathcal{H}^{k}(S\cap B_{1}(0)\cap B_{r}(x))\leq C_{DR}(n)r^{k}, \quad\forall x\in B_{1}(0),~0<r\leq 1.\]
\end{coro}
\section{The \texorpdfstring{$L^{2}$}{$L^{2}$} best approximation estimate}
In this section, we will show that if a ball $B_{r}(x_{0})$ is centered on $\{|U_{\epsilon}(x, 0)|\leq 1-\tau\}$, then  $\beta_{\mu, 2}^{n-1}(x_{0}, r)^{2}$ can be controlled by the energy drop, whenever we are almost one-homogeneous.

We begin with the following result, which is essentially a consequence of the monotonicity formula.
\begin{lem}\label{lem apriori estimate}
There exists a positive constant $\Lambda=\Lambda(\Lambda_{1})$ such that for any $\rho<20$,
\begin{equation}\label{uniform apriori estimate}
\sup_{x_{0}\in B_{1}(0)} \mathbf{\Theta}^{\epsilon}_{U_{\epsilon}}(\rho, x_{0})\leq \Lambda.
\end{equation}
\end{lem}
\begin{proof}
By Lemma \ref{Allen-Cahn monotonicity formula}, we have
\begin{equation}\label{02-27}
\begin{aligned}
\mathbf{\Theta}^{\epsilon}_{U_{\epsilon}}(\rho, x_{0})=&\frac{1}{\rho^{n-2s}}\textbf{E}_{\epsilon}(U_{\epsilon}, \mathcal{B}_{\rho}(\mathbf{x}_{0}))\\
\leq&\frac{1}{39^{n-2s}}\textbf{E}_{\epsilon}(U_{\epsilon}, \mathcal{B}_{39}(\mathbf{x}_{0}))\\
\leq& \frac{1}{39^{n-2s}}\textbf{E}_{\epsilon}(U_{\epsilon}, \mathcal{B}_{40}(\mathbf{0})),
\end{aligned}
\end{equation}
where $\mathbf{x}_{0}=(x_{0}, 0)$, and in the last step   we have applied the condition (H1). Then \eqref{uniform apriori estimate} follows from \eqref{uniform energy assumption'}.
\end{proof}

\begin{lem}\label{lem Not n-symmetry}
There exist positive constants $c\in (0, 1/4)$,$\delta\in (0, 1/4)$ and $\mathbf{k}\geq 1$ depending only on $n, s, \tau, \Lambda_{0}, \Lambda_{1}, W$ such that the following holds.

For every $x_{0}\in B_{1}(0)$, if
$|U_{\epsilon}(x_{0}, 0)|\leq 1-\tau$, $\mathbf{k}\epsilon\leq r<1$ and \[U_{\epsilon}~\text{is}~(0, \delta)-\text{symmetric in}~\mathcal{B}^{+}_{8r}(\mathbf{x}_{0}),\]
where $\mathbf{x}_{0}=(x_{0}, 0)$. Then for any unit orthogonal basis $\{v_{1}, v_{2}, \cdots, v_{n}\}$ of $\mathbb{R}^{n}$,
\begin{equation}\label{lem al1}
\frac{1}{c(n, s, \tau, \Lambda_{0}, \Lambda_{1}, W)}\leq r^{-n+2s}\sum_{i=1}^{n}\int_{\mathcal{B}^{+}_{4r}(\mathbf{x}_{0})\backslash \mathcal{B}^{+}_{3r}(\mathbf{x}_{0})}z^{a}\left((v_{i}, 0)\cdot\nabla U_{\epsilon}\right)^{2}d\mathbf{x}.
\end{equation}
\end{lem}
\begin{proof}
Assume it is false. Then there exist  $\{\epsilon_{j}\}_{j\in\mathbb{N}}\subset (0, 1), \{U_{\epsilon_{j}}\}$ satisfying \eqref{Extension equation}, points $\{x_{j}\}_{j\in\mathbb{N}}\subset B_{1}(0)$, radii $\{r_{j}\}_{j\in\mathbb{N}}\subset (0, 1)$ and a collections of unit orthonormal bases $\{v_{1}^{(j)}, v_{2}^{(j)}, \cdots, v_{n}^{(j)}\}$ such that $\epsilon_{j}/r_{j}\leq 2^{-j}$, $|U_{\epsilon_{j}}(x_{j}, 0)|\leq 1-\tau$,
\[U_{\epsilon_{j}}~\text{is}~(0, 1/j)-\text{symmetric in}~\mathcal{B}^{+}_{8r_{j}}(\mathbf{x}_{j}),\]
where $\mathbf{x}_{j}=(x_{j}, 0)$, but for each $j$,
\begin{equation}
r_{j}^{-n+2s}\sum_{i=1}^{n}\int_{\mathcal{B}^{+}_{4r_{j}}(\mathbf{x}_{j})\backslash \mathcal{B}^{+}_{3r_{j}}(\mathbf{x}_{j})}z^{a}((v_{i}^{(j)}, 0)\cdot\nabla U_{\epsilon_{j}})^{2}d\mathbf{x}\leq 1/j.
\end{equation}

For each $j$, we set $\tilde{\epsilon}_{j}=\epsilon_{j}/r_{j}$ and $\tilde{U}_{\tilde{\epsilon}_{j}}(\mathbf{x})=U_{\epsilon_{j}}(\mathbf{x}_{j}+r_{j}\mathbf{x})$. Then $|\tilde{U}_{\tilde{\epsilon}_{j}}(\mathbf{0})|\leq 1-\tau$,
\begin{equation}\label{quantitative symmetry assumption}
\tilde{U}_{\tilde{\epsilon}_{j}}~\text{is}~(0, 1/j)- \text{symmetric in}~\mathcal{B}^{+}_{8}(\mathbf{0})
\end{equation}
and
\begin{equation}\label{2025-02-09}
\sum_{i=1}^{n}\int_{\mathcal{B}^{+}_{4}(\mathbf{0})\backslash \mathcal{B}^{+}_{3}(\mathbf{0})}z^{a}((v_{i}^{(j)}, 0)\cdot\nabla\tilde{U}_{\tilde{\epsilon}_{j}})^{2}d\mathbf{x}\leq 1/j.
\end{equation}
Assume
\[v_{i}^{(j)}\to v_{i},\quad i=1, 2, \cdots, n.\]
Then $\{v_{1}, v_{2}, \cdots, v_{n} \}$ is also a unit orthonormal basis of $\mathbb{R}^{n}$. By Lemma \ref{lem apriori estimate} and
Theorem \ref{convergence}, there exists a
 function $U_{\ast}$ such that $\tilde{U}_{\tilde{\epsilon}_{j}}\to U_{\ast}$ strongly in $H^{1}(\mathcal{B}_{8}^{+}(\mathbf{0}), z^{a}d\mathbf{x})$. Combining Lemma \ref{lem symmetry of limitation} and \eqref{quantitative symmetry assumption}, we know that $U_{\ast}$ is homogeneous about the origin. Taking limit  in \eqref{2025-02-09}, we  get
\[\sum_{i=1}^{n}\int_{\mathcal{B}^{+}_{4}(\mathbf{0})\backslash \mathcal{B}^{+}_{3}(\mathbf{0})}z^{a}((v_{i}, 0)\cdot\nabla U_{\ast})^{2}d\mathbf{x}=0,\]
and hence $U_{\ast}$ is $n-$ symmetric in $\mathcal{B}^{+}_{8}(\mathbf{0})$. By Remark \ref{n-symmetric nonlocal minimal cone}, we deduce that either $U_{\ast}=1$ or $U_{\ast}=-1$. Without loss of generality, we may assume that $U_{\ast}=1$ on $\mathcal{B}^{+}_{8}(\mathbf{0})$. Then Theorem \ref{convergence} and Lemma \ref{Clear out} tells that $\tilde{U}_{\tilde{\epsilon}_{j}}\to 1$ uniformly on $\mathcal{B}^{+}_{1}(\mathbf{0})$. In particular, $\tilde{U}_{\tilde{\epsilon}_{j}}(\mathbf{0})\to 1$ as $j\to\infty$, which contradicts our assumption that $|\tilde{U}_{\tilde{\epsilon}_{j}}(\mathbf{0})|\leq 1-\tau$.
 \end{proof}
\begin{thm}\label{thm Beta number estimate}
There exist positive constants $c\in (0, 1/4)$,$\delta\in (0, 1/4)$ and $\mathbf{k}\geq 1$ depending only on $n, s, \tau, \Lambda_{0}, \Lambda_{1},W$ such that the following holds.

For every $x_{0}\in B_{1}(0)$, if
$|U_{\epsilon}(x_{0}, 0)|\leq 1-\tau$, $\mathbf{k}\epsilon\leq r<1$ and
\[U_{\epsilon}~\text{is}~(0, \delta)-\text{symmetric in}~\mathcal{B}^{+}_{8r}(\mathbf{x}_{0}),\]
where $\mathbf{x}_{0}=(x_{0}, 0)$, then for any locally finite Borel measure $\mu$ on $\mathbb{R}^{n}$, we have
\begin{equation}\label{thm Beta number estimate2}
\beta_{\mu, 2}^{n-1}(x_{0}, r)^{2}\leq\frac{c(n, s, \tau, \Lambda_{0}, \Lambda_{1}, W)}{r^{n-1}}\int_{B_{r}(x_{0})}\left[\mathbf{\Theta}_{U_{\epsilon}}^{\epsilon}(8r, y)-\mathbf{\Theta}_{U_{\epsilon}}^{\epsilon}(r, y)\right]d\mu(y).
\end{equation}
\end{thm}
\begin{proof}
We divide the proof of \eqref{thm Beta number estimate2}  into the following four steps.

\textbf{Step 1.} The explicit formula of  $\beta_{\mu, 2}^{n-1}(x_{0}, r)^{2}$.

Let
\[x_{cm}=\frac{1}{\mu(B_{r}(x_{0}))}\int_{B_{r}(x_{0})} x d\mu(x)\]
be the center of mass of $B_{r}(x_{0})$. Define the non-negative bilinear form
\[Q(v, w)\equiv \frac{1}{\mu(B_{r}(x_{0}))}\int_{B_{r}(x_{0})}\big(v
\cdot(x-x_{cm})\big)\big(w\cdot(x-x_{cm})\big)d\mu(x).\]
Let $\lambda_{1}\geq \lambda_{2}\geq\cdots\geq \lambda_{n}\geq 0$ be the eigenvalues of the bilinear form $Q$ and let $v_{1}, v_{2},\cdots v_{n}$ be the associated  unit eigenvectors. Then $\{v_{1}, v_{2},\cdots, v_{n}\}$ forms a unit orthogonal basis of $\mathbb{R}^{n}$. By \cite[Section 7.2]{Naber-Valtorta2017},  $\beta_{\mu, 2}^{n-1}(x_{0}, r)^{2}$ is achieved by the affine plane
\[V_{\mu, 2}^{n-1}(x_{0}, r)=x_{cm}+{\rm span}\{v_{1},\cdots, v_{n-1}\}\]
and
\begin{equation}\label{beta number explicit formula}
\beta_{\mu, 2}^{n-1}(x_{0}, r)^{2}=\frac{\mu(B_{r}(x_{0}))}{r^{n+1}}\lambda_{n}.
\end{equation}

\textbf{Step 2.} A different formulation of  $\beta_{\mu, 2}^{n-1}(x_{0}, r)^{2}$.

We identity $\mu$ as  a Borel measure on $\partial\mathbb{R}^{n+1}_{+}:=\mathbb{R}^{n}$. The Borel measure $\mu$ induces a Borel measure $\mu^{e}$ on $\mathbb{R}^{n+1}$, which can be defined in the following way:
\begin{itemize}
\item  if a Borel set $E\cap\partial\mathbb{R}^{n+1}_{+}=\emptyset,$ then $\mu^{e}(E)=0$,
\item if a Borel set $E\cap\partial\mathbb{R}^{n+1}_{+}\neq\emptyset,$ then $\mu^{e}(E)=\mu(E\cap \partial\mathbb{R}^{n+1}_{+})$.
\end{itemize}
 By the definition of $\mu^{e}$, it is easy to check that $\mu(B_{r}(x_{0}))=\mu^{e}(\mathcal{B}_{r}(\mathbf{x}_{0})).$ Let
\[\mathbf{x}_{cm}=\fint_{\mathcal{B}_{r}(\mathbf{x}_{0})} \mathbf{x} d\mu^{e}(\mathbf{x})\]
be the $\mu^{e}-$ center of mass of $\mathcal{B}_{r}(\mathbf{x}_{0})$, then $\mathbf{x}_{cm}=(x_{cm}, 0)$. Define the non-negative bilinear form
\[\textbf{Q}(\textbf{v}, \textbf{w})\equiv \fint_{\mathcal{B}_{r}(\mathbf{x}_{0})}\big(\textbf{v}
\cdot(\textbf{x}-\textbf{x}_{cm})\big)\big(\textbf{w}\cdot(\textbf{x}-\textbf{x}_{cm})\big) d\mu^{e}(\textbf{x}).\]
Here
 \[\fint_{\mathcal{B}_{r}(\mathbf{x}_{0})}f(\mathbf{x})d\mu^{e}(\mathbf{x})=\frac{1}{\mu^{e}(\mathcal{B}_{r}(\mathbf{x}))}\int_{\mathcal{B}_{r}(\mathbf{x}_{0}) }f(\mathbf{x})d\mu^{e}(\mathbf{x})\]
is the integral average of a function $f$. Let $\hat{\lambda}_{1}\geq \hat{\lambda}_{2}\geq\cdots\geq \hat{\lambda}_{n}\geq \hat{\lambda}_{n+1}\geq0$ be the eigenvalues of the bilinear form $\textbf{Q}$ and let $\mathbf{v}_{1}, \mathbf{v}_{2},\cdots \mathbf{v}_{n}, \mathbf{v}_{n+1}$ be the associated unit eigenvectors.

We claim: $\hat{\lambda}_{n+1}=0$ and
\[\hat{\lambda}_{i}=\lambda_{i}, \quad\text{for}~i=1,2,\cdots, n.\]

Indeed, by denoting $e_{n+1}=(0, 0, \cdots, 1)$, then for any $\textbf{w}\in\mathbb{R}^{n+1}$,
\[\textbf{Q}(e_{n+1}, \textbf{w})\equiv \fint_{\mathcal{B}_{r}(\mathbf{x}_{0})}\big(e_{n+1}
\cdot(\textbf{x}-\textbf{x}_{cm})\big)\big(\textbf{w}\cdot(\textbf{x}-\textbf{x}_{cm})\big) d\mu^{e}(\textbf{x})=0.\]
Because the quadratic form $\textbf{Q}$ is nonnegative,  $0$ is the lowest eigenvalue and $e_{n+1}$ is an eigenvector associated to the eigenvalue $0$. Since the different eigenvectors are orthogonal,  for each $i\in\{1, 2, \cdots, n\}$, there exists $\tilde{v}_{i}\in\mathbb{R}^{n}$ such that $\mathbf{v}_{i}=(\tilde{v}_{i}, 0)$. For any $w\in\mathbb{R}^{n}$,  we have
\begin{align}
\hat{\lambda}_{i}(\mathbf{v}_{i}\cdot(w, 0))&=\mathbf{Q}(\mathbf{v}_{i}, (w, 0))\notag\\
&=\fint_{\mathcal{B}_{r}(\mathbf{x}_{0})}\big(\textbf{v}_{i}
\cdot(\textbf{x}-\textbf{x}_{cm})\big)\big((w, 0)\cdot(\textbf{x}-\textbf{x}_{cm})\big) d\mu^{e}(\textbf{x})\notag\\
&=\fint_{B_{r}(x_{0})}\big(\tilde{v}_{i}
\cdot(x-x_{cm})\big)\big(w\cdot(x-x_{cm})\big) d\mu(x)\notag\\
&=Q(\tilde{v}_{i}, w).\notag
\end{align}
Hence $\mathbf{v}_{i}$ is an eigenvector of $\mathbf{Q}$
if and only if $\tilde{v}_{i}$ is an eigenvector of $Q$. Moreover, $\hat{\lambda}_{i}$ is an eigenvalue of $\mathbf{Q}$
if and only if $\hat{\lambda}_{i}$ is an eigenvalue of $Q$. Then the claim follows from Step 1.

\textbf{Step 3.} We claim that for each $i\in\{1,2,\cdots n\}$ and  $\mathbf{x}=(x, z)\in\mathbb{R}^{n+1}_{+}$,
\begin{equation}\label{the5.4.1}
\hat{\lambda_{i}}z^{a}( \mathbf{v}_{i}\cdot\nabla U_{\epsilon}(\mathbf{x}))^{2}\leq \fint_{\mathcal{B}_{r}(\textbf{x}_{0})}z^{a}|(\textbf{y}-\textbf{x})\cdot\nabla U_{\epsilon}(\textbf{x})|^{2}d\mu^{e}(\textbf{y}).
\end{equation}
By using the definition of center of mass,
\begin{align}
\hat{\lambda}_{i}z^{\frac{a}{2}}(\mathbf{v}_{i}\cdot \nabla U_{\epsilon}(\mathbf{x}))&=\textbf{Q}(\mathbf{v}_{i}, z^{\frac{a}{2}}\nabla U_{\epsilon}(\mathbf{x}))\notag\\
&=\fint_{\mathcal{B}_{r}(\textbf{x}_{0})}\Big(\mathbf{v}_{i}\cdot(\textbf{y}-\textbf{x}_{cm})\Big)\Big(z^{\frac{a}{2}}\nabla U_{\epsilon}(\mathbf{x})\cdot(\textbf{y}-\textbf{x}_{cm})\Big)d\mu^{e}(\textbf{y})\notag\\
&=\fint_{\mathcal{B}_{r}(\textbf{x}_{0})}\Big(\mathbf{v}_{i}\cdot(\textbf{y}-\textbf{x}_{cm})\Big)\Big(z^{\frac{a}{2}}(\textbf{y}-\textbf{x})\cdot\nabla U_{\epsilon}(\textbf{x})\Big)d\mu^{e}(\textbf{y})\notag\\
&\leq\hat{\lambda}_{i}^{\frac{1}{2}}\Big(\fint_{\mathcal{B}_{r}(\textbf{x}_{0})}z^{a}|(\textbf{y}-\textbf{x})\cdot\nabla U_{\epsilon}(\textbf{x})|^{2}d\mu^{e}(\textbf{y})\Big)^{\frac{1}{2}}.\notag
\end{align}
This implies \eqref{the5.4.1}.

\textbf{Step 4.} The proof of \eqref{thm Beta number estimate2}.

We calculate
\begin{align}\label{beta number key}
&\hat{\lambda}_{i}r^{-n-2+2s}\int_{\mathcal{B}^{+}_{4r}(\mathbf{x}_{0})\backslash \mathcal{B}^{+}_{3r}(\mathbf{x}_{0})}z^{a}(\mathbf{v}_{i}\cdot\nabla U_{\epsilon}(\mathbf{x}))^{2}d\mathbf{x}\notag\\
\leq&r^{-n-2+2s}\int_{\mathcal{B}^{+}_{4r}(\mathbf{x}_{0})\backslash \mathcal{B}^{+}_{3r}(\mathbf{x}_{0})}\fint_{\mathcal{B}_{r}(\mathbf{x}_{0})}|(\textbf{y}-\textbf{x})\cdot\nabla U_{\epsilon}(\textbf{x})|^{2}d\mu^{e}(\textbf{y})d\textbf{x}\notag\\
\leq& c(n)\fint_{\mathcal{B}_{r}(\mathbf{x}_{0})}\int_{\mathcal{B}^{+}_{4r}(\mathbf{x}_{0})\backslash \mathcal{B}^{+}_{3r}(\mathbf{x}_{0})}z^{a}\frac{|(\textbf{y}-\textbf{x})\cdot\nabla U_{\epsilon}(\textbf{x})|^{2}}{|\textbf{y}-\textbf{x}|^{n+2-2s}}d\mu^{e}(\textbf{y})d\textbf{x}\notag\\
\leq& c(n)\fint_{\mathcal{B}_{r}(\mathbf{x}_{0})}\int_{\mathcal{B}^{+}_{4r}(\mathbf{x}_{0})\backslash \mathcal{B}^{+}_{3r}(\mathbf{x}_{0})}z^{a}\frac{|(\textbf{y}-\textbf{x})\cdot\nabla U_{\epsilon}(\textbf{x})|^{2}}{|\textbf{y}-\textbf{x}|^{n+2-2s}}d\textbf{x}d\mu^{e}(\textbf{y})\\
\leq& c(n)\fint_{B_{r}(x_{0})}\int_{\mathcal{B}^{+}_{4r}(\mathbf{x}_{0})\backslash \mathcal{B}^{+}_{3r}(\mathbf{x}_{0})}z^{a}\frac{|((y, 0)-\textbf{x})\cdot\nabla U_{\epsilon}(\textbf{x})|^{2}}{|(y, 0)-\textbf{x}|^{n+2-2s}}d\textbf{x}d\mu(y)\notag\\
\leq& c(n)\fint_{B_{r}(x_{0})}\int_{\mathcal{B}^{+}_{8r}((y, 0))\backslash \mathcal{B}^{+}_{r}((y, 0))}z^{a}\frac{|((y, 0)-\textbf{x})\cdot\nabla U_{\epsilon}(\textbf{x})|^{2}}{|(y, 0)-\textbf{x}|^{n+2-2s}}d\textbf{x}d\mu(y)\notag\\
\leq &c(n)\fint_{B_{r}(x_{0})}[\mathbf{\Theta}_{U_{\epsilon}}^{\epsilon}(8r, y)-\mathbf{\Theta}_{U_{\epsilon}}^{\epsilon}(r, y)]d\mu(y).\notag
\end{align}
Now we fix  $\delta=\delta(n, s, \tau, \Lambda_{0}, \Lambda_{1}, W)$ and $\mathbf{k}(n, s, \tau, \Lambda_{0}, \Lambda_{1}, W)$ satisfying Lemma \ref{lem Not n-symmetry}.
Combining \eqref{the5.4.1} and Lemma \ref{lem Not n-symmetry},  we deduce
\begin{equation*}\label{the5.4.7}
\begin{aligned}
&\beta_{\mu, 2}^{n-1}(x_{0}, r)^{2}\\
\overset{\eqref{beta number explicit formula}}{\leq}&\frac{\mu(B_{r}(x_{0}))}{r^{n+1}}\sum_{i=1}^{n}\lambda_{i}\\
\overset{\eqref{lem al1}}{\leq}&\frac{\mu(B_{r}(x_{0}))}{r^{n+1}}c(n, s, \tau, \Lambda_{0}, \Lambda_{1}, W)\sum\limits_{i=1}^{n}\frac{\hat{\lambda}_{i}}{ r^{n-2s}}\int_{\mathcal{B}^{+}_{4r}(\mathbf{x}_{0})\backslash \mathcal{B}^{+}_{3r}(\mathbf{x}_{0})}z^{a}(\mathbf{v}_{i}\cdot\nabla U_{\epsilon}(\mathbf{x}))^{2}d\mathbf{x}\\
\overset{\eqref{beta number key}}{\leq}&\frac{c(n, s, \tau, \Lambda_{0}, \Lambda_{1}, W)}{r^{n-1}}\int_{B_{r}(x_{0})}\left[\mathbf{\Theta}_{U_{\epsilon}}^{\epsilon}(8r, y)-\mathbf{\Theta}_{U_{\epsilon}}^{\epsilon}(r, y)\right]d\mu(y).
\end{aligned}
\end{equation*}
This completes the proof of Theorem \ref{thm Beta number estimate}.
\end{proof}

\section{Centered density drop gives packing}
\setcounter{equation}{0}
In this section, our main objective is to demonstrate that if we have a collection of disjoint balls with small energy drop at the center, then the Discrete Reifenberg Theorem gives good packing estimates.

First, we prove that small energy drops imply almost-homogeneous.
\begin{lem}\label{Quantutative homogeneous}
Take $\delta>0$. There exist constants $\gamma\in (0, 1/4)$ and $\mathbf{k}\geq 1$ depending only on $n, s, \tau, \Lambda_{0}, \Lambda_{1}, \delta,W$ satisfy the following condition.

For each $x\in B_{1}(0)$  and $r<20$, if $\mathbf{k}\epsilon\leq r$ and
\[\mathbf{\Theta}_{U_{\epsilon}}^{\epsilon}(r, x)-\mathbf{\Theta}_{U_{\epsilon}}^{\epsilon}(\gamma r, x)\leq \gamma,\]
then $U_{\epsilon}$ is $(0, \delta)-$symmetric in $\mathcal{B}^{+}_{r}(\mathbf{x}).$
\end{lem}
\begin{proof}
Assume it is false for some  constant $\delta$, then there exist sequences $\{\epsilon_{j}\}_{j\in\mathbb{N}}\subset (0, 1), \{U_{\epsilon_{j}}\}$ satisfying \eqref{Extension equation}, points $\{x_{j}\}_{j\in\mathbb{N}}\subset B_{1}(0)$, radii $\{r_{j}\}_{j\in\mathbb{N}}\subset (0, 20)$ and $\gamma_{j}\to 0$ such that $\epsilon_{j}/r_{j}\leq 2^{-j}$,
\[\mathbf{\Theta}_{U_{\epsilon_{j}}}^{\epsilon_{j}}(r_{j}, x_{j})-\mathbf{\Theta}_{U_{\epsilon_{j}}}^{\epsilon_{j}}(\gamma_{j} r_{j}, x_{j})\leq \gamma_{j},\]
but
\[U_{\epsilon_{j}}~\text{is not}~(0, \delta)-\text{symmetric in}~\mathcal{B}^{+}_{r_{j}}(\mathbf{x}_{j}),\]
where $\mathbf{x}_{j}=(x_{j}, 0)$. For each $j$, we set $\tilde{\epsilon}_{j}=\epsilon_{j}/r_{j}$ and $\tilde{U}_{\tilde{\epsilon}_{j}}(\mathbf{x})=U_{\epsilon_{j}}(\mathbf{x}_{j}+r_{j}\mathbf{x})$. Now our assumptions lead to
\begin{equation}\label{contradiction assumption}
\mathbf{\Theta}_{\tilde{U}_{\tilde{\epsilon}_{j}}}^{\tilde{\epsilon}_{j}}(1, 0)-\mathbf{\Theta}_{\tilde{U}_{\tilde{\epsilon}_{j}}}^{\tilde{\epsilon}_{j}}(\gamma_{j}, 0)\leq \gamma_{j},
\end{equation}
but
\[\tilde{U}_{\tilde{\epsilon}_{j}}~\text{is not}~(0, \delta)-\text{symmetric in}~\mathcal{B}^{+}_{1}(\mathbf{0}).\]
By Lemma \ref{lem apriori estimate} and Theorem \ref{convergence}, we can assume $\tilde{U}_{\tilde{\epsilon}_{j}}\to U_{\ast}$ strongly, where $U_{\ast}$ is a solution of \eqref{eqn1} satisfying \eqref{eqn2}. Combing Theorem  \ref{convergence}, Lemma \ref{Uppersemicontinuity} and \eqref{contradiction assumption}, we have
\[\mathbf{\Theta}_{U_{\ast}}(1,0)-\mathbf{\Theta}_{U_{\ast}}(0)\leq \limsup\limits_{j\to\infty}\Big[\mathbf{\Theta}_{\tilde{U}_{\tilde{\epsilon}_{j}}}^{\tilde{\epsilon}_{j}}(1, 0)-\mathbf{\Theta}_{\tilde{U}_{\tilde{\epsilon}_{j}}}^{\tilde{\epsilon}_{j}}(\gamma_{j}, 0)\Big]=0.\]
By Lemma \ref{lem Monotonicityformula}, $\mathbf{\Theta}_{U_{\ast}}(r, 0)$ is nondecreasing with respect to $r$. It follows that $\mathbf{\Theta}_{U_{\ast}}(\cdot, 0)$ is constant in $(0, 1)$ and $U_{\ast}$ is homogeneous with respect to $\mathbf{0}$ in $\mathcal{B}_{1}^{+}(\mathbf{0})$.
Using the strong convergence  once again, we obtain
\[\int_{\mathcal{B}^{+}_{1}(\mathbf{0})}z^{a}(\tilde{U}_{\tilde{\epsilon}_{j}}-U_{\ast})^{2}d\mathbf{x}\to 0\]
as $j\to\infty$. This implies $\tilde{U}_{\tilde{\epsilon}_{j}}$ is $(0, \delta)$-symmetric in $\mathcal{B}^{+}_{1}(\mathbf{0})$  for $j$ large enough, which contradicts our assumption.
\end{proof}
Next, we prove two properties on the packing measures.
\begin{lem}\label{lempac}
Let $\{B_{r_{p}}(p)\}_{p}$ be a collection of disjoint balls, with $p\in B_{1}(0)$
 and
 \[0\leq r_{p}\leq 1,\quad\forall p.\]
 For each  $i\in\mathbb{N}$ and $k\in\{1, 2, \cdots, n-1\}$, define the packing measure
\begin{equation}\label{defmu}
\mu_{i}=\sum\limits_{r_{p}\leq 2^{-i}}r_{p}^{k}\delta_{p}.
\end{equation}
Suppose $x\in\rm{spt}\mu_{i}$ and $j\geq i$, then
\begin{equation}\label{lem6.1.6}
\beta_{\mu_{i}, 2}^{k}(x, 2^{-j})=\left\{\begin{array}{lll}
\beta_{\mu_{j}, 2}^{k}(x, 2^{-j}),\quad&\text{if}~x\in\rm{spt}\mu_{j},\\
0,&\text{otherwise}.
\end{array}
\right.
\end{equation}
\end{lem}
\begin{proof}
For convenience, we denote $r_{i}=2^{-i}$. Since $x\in{\rm spt}\mu_{i}$, $r_{x}\leq r_{i}$. First, we assume that $x\not\in{\rm spt}\mu_{j}$, so $r_{x}>r_{j}$. Since the balls $B_{r_{p}}(p)$ are disjoint,
\[{\rm spt}\mu_{i}\cap B_{r_{j}}(x)=x.\]
By choosing a $k-$affine plane passing through $x$, we know that $\beta_{\mu_{i}, 2}^{k}(x, r_{j})=0$.

Next, we assume $x\in{\rm spt}\mu_{j}$. By the definition, we have $r_{x}\leq r_{j}\leq r_{i}.$ If $y\in B_{r_{j}}(x)\backslash B_{r_{x}}(x)$ and $y\in {\rm spt}\mu_{i}$, then $r_{y}\leq r_{i}$. Since the balls $B_{r_{p}}(p)$ are disjoint, we have
 \[r_{y}\leq r_{j}-r_{x}<r_{j}.\]
 This implies $y\in{\rm spt}\mu_{j}$. In particular, we have proved that
\begin{equation}\label{lempac1}
({\rm spt}\mu_{i}\cap B_{r_{j}}(x))\subset {\rm spt}\mu_{j}.
\end{equation}
On the other hand, it is clear that
\begin{equation}\label{lempac2}
({\rm spt}\mu_{j}\cap B_{r_{j}}(x))\subset{\rm spt}\mu_{i}.
\end{equation}
By \eqref{lempac1} and \eqref{lempac2}, \eqref{defmu} and \eqref{NOT1}, we conclude that
\[\mu_{i}\llcorner B_{x}(r_{j})=\mu_{j} \llcorner B_{x}(r_{j}).\]
By the definition, it follows easily that
\[\beta_{\mu_{i}, 2}^{k}(x, r_{j})=\beta_{\mu_{j}, 2}^{k}(x, r_{j}). \qedhere\]
\end{proof}
\begin{lem}\label{lempacc}
Let $\{B_{2r_{p}}(p)\}_{p}$ be a collection of disjoint balls, with $p\in B_{1}(0)$
 and
 \[0\leq r_{p}\leq 1,\quad\forall p.\]
 For each integer $i\in\mathbb{N}$ and $k\in\{1, 2, \cdots, n-1\}$, let $\mu_{i}$ be defined by \eqref{defmu}. Then for any $x\in B_{1}(0)$,
\begin{equation}\label{lempacc1}
\sum\limits_{r_{j}\leq 2r_{i}}\int_{B_{r_{i}}(x)}\beta_{\mu_{i}, 2}^{k}(y, r_{j})^{2}d\mu_{i}(y)
=\sum\limits_{r_{j}\leq 2r_{i}}\int_{B_{r_{i}}(x)}\beta_{\mu_{j}, 2}^{k}(y, r_{j})^{2}d\mu_{j}(y).
\end{equation}
\end{lem}
\begin{proof}
By the definition of $\mu_{i}$ and Lemma \ref{lempac}, we have
\[
\begin{aligned}
&\sum\limits_{r_{j}\leq 2r_{i}}\int_{B_{r_{i}}(x)}\beta_{\mu_{i}, 2}^{k}(y, r_{j})^{2}d\mu_{i}(y)\\
=&\sum\limits_{r_{j}\leq 2r_{i}}\sum\limits_{p\in B_{r_{i}}(x), r_{p}\leq r_{i}}\beta_{\mu_{i}, 2}^{k}(p, r_{j})^{2}r_{p}^{k}\\
=&\sum\limits_{r_{j}\leq r_{i}}\sum\limits_{p\in B_{r_{i}}(x), r_{p}\leq r_{i}}\beta_{\mu_{i}, 2}^{k}(p, r_{j})^{2}r_{p}^{k}+\sum\limits_{r_{i}<r_{j}\leq 2r_{i}}\sum\limits_{p\in B_{r_{i}}(x), r_{p}\leq r_{i}}\beta_{\mu_{i}, 2}^{k}(p, r_{j})^{2}r_{p}^{k}\\
=&\sum\limits_{r_{j}\leq r_{i}}\sum\limits_{p\in B_{r_{i}}(x), r_{p}\leq r_{j}}\beta_{\mu_{j}, 2}^{k}(p, r_{j})^{2}r_{p}^{k}+\sum\limits_{r_{i}<r_{j}\leq 2r_{i}}\sum\limits_{p\in B_{r_{i}}(x), r_{p}\leq r_{i}}\beta_{\mu_{i}, 2}^{k}(p, r_{j})^{2}r_{p}^{k}.
\end{aligned}\]
Notice that
\[\begin{aligned}
&\sum\limits_{r_{i}<r_{j}\leq 2r_{i}}\sum\limits_{p\in B_{r_{i}}(x), r_{p}\leq r_{j}}\beta_{\mu_{j}, 2}^{k}(p, r_{j})^{2}r_{p}^{k}\\
=&\sum\limits_{r_{i}<r_{j}\leq 2r_{i}}\sum\limits_{p\in B_{r_{i}}(x), 2r_{p}\leq r_{j}}\beta_{\mu_{j}, 2}^{k}(p, r_{j})^{2}r_{p}^{k}+\sum\limits_{r_{i}<r_{j}\leq 2r_{i}}\sum\limits_{p\in B_{r_{i}}(x), r_{p}\leq r_{j}<2r_{p}}\beta_{\mu_{j}, 2}^{k}(p, r_{j})^{2}r_{p}^{k}.
\end{aligned}\]

\textbf{Case 1}: $p\in B_{r_{i}}(x)$ and $r_{p}\leq r_{j}<2r_{p}$.

Since the balls $\{B_{2r_{p}}(p)\}$ are disjoint,
\[\rm {spt} \mu_{i}\cap B_{r_{j}}(p)=\rm {spt} \mu_{j}\cap B_{r_{j}}(p)=\{p\}.\]
By choosing a $k-$affine plane passing through $p$, we get \[\beta_{\mu_{i}, 2}^{k}(p, r_{j})=\beta_{\mu_{j}, 2}^{k}(p, r_{j})=0.\]

\textbf{Case 2}: $p\in B_{r_{i}}(x)$ and $2r_{p}\leq r_{j}\leq2r_{i}$.

If there exists a point $p'\in B_{r_{j}}(p)\cap \rm {spt}\mu_{j}$, then
\[2r_{p}+2r_{p'}<r_{j}\leq 2r_{i}.\]
It follows that $r_{p}+r_{p'}<r_{i}$ and $p'\in spt \mu_{i}$. Therefore, we have shown that
\[\mu_{i}\llcorner B_{r_{j}}(p)=\mu_{j}\llcorner B_{r_{j}}(p).\]
Hence
\[\begin{aligned}
\sum\limits_{r_{i}<r_{j}\leq 2r_{i}}\sum\limits_{p\in B_{r_{i}}(x), r_{p}\leq r_{j}}\beta_{\mu_{j}, 2}^{k}(p, r_{j})^{2}r_{p}^{k}
=&\sum\limits_{r_{i}<r_{j}\leq 2r_{i}}\sum\limits_{p\in B_{r_{i}}(x), 2r_{p}\leq r_{j}}\beta_{\mu_{j}, 2}^{k}(p, r_{j})^{2}r_{p}^{k}\\
=&\sum\limits_{r_{i}<r_{j}\leq 2r_{i}}\sum\limits_{p\in B_{r_{i}}(x), r_{p}\leq r_{i}}\beta_{\mu_{i}, 2}^{k}(p, r_{j})^{2}r_{p}^{k}.
\end{aligned}\]
Combining Case 1 and Case 2, we can get
\begin{equation}\label{claim}
\sum\limits_{r_{i}<r_{j}\leq 2r_{i}}\sum\limits_{p\in B_{r_{i}}(x), r_{p}\leq r_{i}}\beta_{\mu_{i}, 2}^{k}(p, r_{j})^{2}r_{p}^{k}=\sum\limits_{r_{i}<r_{j}\leq 2r_{i}}\sum\limits_{p\in B_{r_{i}}(x), r_{p}\leq r_{j}}\beta_{\mu_{j}, 2}^{k}(p, r_{j})^{2}r_{p}^{k}.
\end{equation}
Then \eqref{lempacc1} follows from \eqref{claim}.
\end{proof}
\begin{lem}\label{key packing}
There exist  $\eta\in (0, 1)$ and $\mathbf{k}\geq 1$ depending only on $n, s, \tau,\Lambda_{0}, \Lambda_{1}, W$ such that the following holds.

Assume $r\in (0, 1), \mathbf{k}\epsilon<r$ and
 \[\sup\limits_{x\in B_{1}(0)}\mathbf{\Theta}_{U_{\epsilon}}^{\epsilon}(2, x)\leq M.\]
 If $\{B_{2r_{p}}(p)\}_{p}$ is a collection of disjoint balls satisfying
 \begin{equation}\label{key packing1}
 \mathbf{\Theta}_{U_{\epsilon}}^{\epsilon}(\eta r_{p},p)\geq M-\eta,\quad p\in \big\{|U_{\epsilon}(x, 0)|\leq 1-\tau\big\}\cap B_{1}(0),\quad r\leq r_{p}\leq 1,
 \end{equation}
 then
 \begin{equation}\label{key packing2}
 \sum\limits_{p}r_{p}^{n-1}\leq c(n).
 \end{equation}
\end{lem}
\begin{proof}
\textbf{Step 1.}  The choice of  $\eta$ and $\mathbf{k}$.

      We choose  $\delta(n, s, \tau, \Lambda_{0}, \Lambda_{1},W)$ and $\mathbf{k}_{1}(n, s, \tau, \Lambda_{0}, \Lambda_{1}, W)$ satisfying Theorem \ref{key packing}. According to Lemma \ref{Quantutative homogeneous}, there exist   $\gamma(n, s, \tau, \Lambda_{0}, \Lambda_{1},W)$ and $\mathbf{k}_{2}(n, s, \tau, \Lambda_{0}, \Lambda_{1},W)$ such that if
\[\mathbf{\Theta}_{U_{\epsilon}}^{\epsilon}(8r_{p}, p)-\mathbf{\Theta}_{U_{\epsilon}}^{\epsilon}(\gamma r_{p}, p)<\gamma,\]
then $U_{\epsilon}$ is $(0, \delta)-$ symmetric in $\mathcal{B}^{+}_{8r_{p}}(\mathbf{p})$. Choose
\[\eta(n, s, \tau, \Lambda_{0}, \Lambda_{1},W)=\min\{\delta (n, s, \tau, \Lambda_{0}, \Lambda_{1},W), \gamma (n, s, \tau, \Lambda_{0}, \Lambda_{1},W)\}\]
and
\[\mathbf{k}(n, s, \tau, \Lambda_{0}, \Lambda_{1},W)=\max\{\mathbf{k}_{1}(n, s, \tau, \Lambda_{0}, \Lambda_{1},W), \mathbf{k}_{2}(n, s, \tau, \Lambda_{0}, \Lambda_{1},W)\}.\]

For each integer $i\in\mathbb{N}$, we define the packing measure
\begin{equation}\label{lem6.1.4}
\mu_{i}=\sum_{r_{p}\leq r_{i}}r_{p}^{n-1}\delta_{p}.
\end{equation}
Clearly, the required estimate \eqref{key packing2} is equivalent to
\begin{equation}\label{lem6.1.5}
\mu_{0}(B_{1}(0))\leq c(n).
\end{equation}
Notice that if $r_{p}\leq 2^{-2}$,
\[\mathbf{\Theta}_{U_{\epsilon}}^{\epsilon}(8r_{p}, p)-\mathbf{\Theta}_{U_{\epsilon}}^{\epsilon}(\eta r_{p},p)\leq M-\mathbf{\Theta}_{U_{\epsilon}}^{\epsilon}(\eta r_{p},p)\leq\eta,\]
By Lemma \ref{Quantutative homogeneous}, Theorem \ref{key packing} and the choice of $\eta$, we have
\begin{equation}\label{lem6.1.7}
\beta_{\mu_{i}, 2}^{n-1}(p, r_{p})^{2}\leq\frac{c(n, s, \tau,\Lambda_{0}, \Lambda_{1},W)}{r_{p}^{n-1}} \int_{B_{r_{p}}(p)}\big[\mathbf{\Theta}_{U_{\epsilon}}^{\epsilon}(8r_{p}, y)-\mathbf{\Theta}_{U_{\epsilon}}^{\epsilon}(r_{p}, y)\big]d\mu_{i}(y)
\end{equation}
whenever $r_{p}\leq 2^{-2}$.

\textbf{Step 2.} A packing estimate based on induction.

 For $r_{i}\leq 2^{-4}$, we shall inductively prove the estimate
\begin{equation}\label{lem6.1.8}
\mu_{i}(B_{r_{i}}(x))\leq C_{DR}(n)r_{i}^{n-1},\quad\forall x\in B_{1}(0).
\end{equation}
Here $C_{DR}(n)$ is the constant from Theorem \ref{theRec1}.

 If $i$ is large enough, then $r_{i}<r$. By \eqref{key packing1} and \eqref{lem6.1.4}, we know that $\mu_{i}=0$. Therefore, \eqref{lem6.1.8} holds trivially provided that $i$ is sufficiently large.

 Let us suppose the inductive hypothesis that \eqref{lem6.1.8} holds for all $j\geq i+1$. Fix a point $x\in B_{1}(0)$, we have
\begin{equation}
\mu_{j}(B_{4r_{j}}(x))=\mu_{j+2}(B_{4r_{j}}(x))+\sum r_{p}^{n-1},
\end{equation}
where the sum is over $p\in B_{4r_{j}}(x)$ with $r_{j+2}<r_{p}\leq r_{j}$. Recall that the balls  $B_{2r_{p}}(p)$ are disjoint, so we are summing over at most $c(n)$ points. Therefore,
\begin{equation}\label{lem6.1.9}
\mu_{j}(B_{4r_{j}}(x))\leq\Gamma(n)r_{j}^{n-1},\quad\forall j\geq i-2,\quad\forall x\in B_{1}(0),
\end{equation}
where $\Gamma=c(n)C_{DR}$.

For every $x\in B_{1}(0)$, we calculate,
\begin{align}
&\sum\limits_{r_{j}\leq 2r_{i}}\int_{B_{r_{i}}(x)}\beta_{\mu_{i}, 2}^{n-1}(z, r_{j})^{2}d\mu_{i}(z)\notag\\
\overset{\eqref{lempacc1}}{=}&\sum\limits_{r_{j}\leq 2r_{i}}\int_{B_{r_{i}}(x)}\beta_{\mu_{j}, 2}^{n-1}(z, r_{j})^{2}d\mu_{j}(z)\notag\\
\overset{\eqref{lem6.1.7}}{\leq}& c\sum\limits_{r_{j}\leq 2r_{i}}\frac{1}{r_{j}^{n-1}}\int_{B_{r_{i}}(x)}\int_{B_{r_{j}}(z)}\left[\mathbf{\Theta}_{U_{\epsilon}}^{\epsilon}(8r_{j}, y)-\mathbf{\Theta}_{U_{\epsilon}}^{\epsilon}(r_{j}, y)\right] d\mu_{j}(y)d\mu_{j}(z)\notag\\
\overset{\textrm{Fubini}}{\leq} &c\sum\limits_{r_{j}\leq 2r_{i}}\int_{B_{r_{i}+r_{j}}(x)}\frac{\mu_{j}(B_{r_{j}}(y))}{r_{j}^{n-1}}\left[\mathbf{\Theta}_{U_{\epsilon}}^{\epsilon}(8r_{j}, y)-\mathbf{\Theta}_{U_{\epsilon}}^{\epsilon}(r_{j}, y)\right]d\mu_{j}(y)\notag\\
\overset{\eqref{lem6.1.9}}{\leq}&c\Gamma\sum\limits_{r_{j}\leq 2r_{i}}\int_{B_{r_{i}+r_{j}}(x)}\left[\mathbf{\Theta}_{U_{\epsilon}}^{\epsilon}(8r_{j}, y)-\mathbf{\Theta}_{U_{\epsilon}}^{\epsilon}(r_{j}, y)\right]d\mu_{j}(y)\notag\\
\leq&c\Gamma\int_{B_{4r_{i}}(x)}\sum\limits_{r_{j}\leq 2r_{i}}\left[\mathbf{\Theta}_{U_{\epsilon}}^{\epsilon}(8r_{j}, y)-\mathbf{\Theta}_{U_{\epsilon}}^{\epsilon}(r_{j}, y)\right]d\mu_{j}(y)\notag\\
\leq & c\Gamma\Big\{\sum\limits_{p\in B_{4r_{i}}(x)\cap{\rm spt}\mu_{i}}r_{p}^{n-1}\left[\mathbf{\Theta}_{U_{\epsilon}}^{\epsilon}(16r_{i}, p)-\mathbf{\Theta}_{U_{\epsilon}}^{\epsilon}(r_{p}, p)\right]\Big\}\notag\\
\leq & c\Gamma\eta \mu_{i}(B_{4r_{i}}(x))\notag\\
\leq & c(n, s, \tau,\Lambda_{0}, \Lambda_{1},W)\Gamma^{2}\eta r_{i}^{n-1}.\notag
\end{align}
Ensuring $\eta=\eta (n, s, \tau,\Lambda_{0}, \Lambda_{1})$ sufficiently small, we deduce that
\[\sum\limits_{r_{j}\leq 2r_{i}}\int_{B_{r_{i}}(x)}\beta_{\mu_{i}, 2}^{n-1}(y, r_{j})^{2}d\mu_{i}(y)\leq 2^{-2-2n}\delta_{DR}(n)r_{i}^{n-1}.\]
Therefore, Corollary \ref{correc1} yields
\[\mu_{i}(B_{r_{i}}(x))\leq C_{DR}(n)r_{i}^{n-1}.\]
This proves \eqref{lem6.1.8}. By  induction, we know that \eqref{lem6.1.8} holds for all $r_{i}\leq 2^{-4}$.

\textbf{Step 3.} The proof of \eqref{key packing2}.

It is easy to see there exists a universal constant $N$ such that
$B_{1}(0)\subset \bigcup_{i=1}^{N} B_{r_{4}}(x_{i}),$
where $x_{i}, i=1, \cdots, N$  are some points in $B_{1}(0)$. Therefore,
\[
\begin{aligned}
\mu_{0}(B_{1}(0))&\leq \sum_{i=1}^{N}\mu_{0}(B_{r_{4}}(x_{i}))\\
&=\sum_{i=1}^{N}\Big[\mu_{4}(B_{r_{4}}(x_{i}))+\sum_{p\in B_{r_{4}}(x_{i}),r_{p}>r_{4}}r_{p}^{n-1}\Big]\\\
&\leq c(n),
\end{aligned}\]
here in the last inequality, we have applied Step 2 and the fact that the number of $p$ satisfying ${p\in B_{r_{4}}(x_{i}),r_{p}>r_{4}}$ can be controlled by a universal constant.
\end{proof}

\section{The key dichotomy}
\setcounter{equation}{0}
In this section, we prove a crucial dichotomy which is a restatement of effective dimension reduction in the vein of Federer and Almgren. Roughly speaking, this result says that either we have  small density drop on $\big\{|U_{\epsilon}(x, 0)|\leq 1-\tau\big\}$, or the high-density points are concentrated near an $(n-2)$-dimensional affine plane.
\begin{thm}\label{thm Allen-Cahn Dichotomy}
Given $\gamma>0, \eta'>0,\rho>0$,  there exist $\eta\in (0, 1/4)$ and $\mathbf{k}\geq 1$ depending only on $n, s, \tau, \gamma, \eta', \rho, \Lambda_{0}, \Lambda_{1},W$ such that the following holds.

Suppose $ r<1$ and $\mathbf{k}\epsilon\leq r$. For each $x_{0}\in B_{1}(0)$ such that $|U_{\epsilon}(x_{0}, 0)|\leq 1-\tau$, we set
\[M=\sup_{x\in B_{r}(x_{0})}\mathbf{\Theta}_{U_{\epsilon}}^{\epsilon}(2r, x).\]
Then at least one of the following two possibilities occurs:
\begin{itemize}
\item[(i)] we have
\[\mathbf{\Theta}_{U_{\epsilon}}^{\epsilon}(\gamma\rho r, x)\geq M-\eta',\quad\text{on}~\big\{|U_{\epsilon}(x, 0)|\leq 1-\tau\big\}\cap B_{r}(x_{0}),\]

\item[(ii)] there is an $n-2$ dimensional affine plane $L^{n-2}$ such that \[\{x:\mathbf{\Theta}_{U_{\epsilon}}^{\epsilon}(2\eta r, x)\geq M-\eta\}\cap B_{r}(x_{0})\subset\mathcal{J}_{\rho r}(L^{n-2}).\]
\end{itemize}
\end{thm}
Theorem \ref{thm Allen-Cahn Dichotomy} will be  an easy consequence of the following two lemmas.
\begin{lem}\label{lem Allen-Cahn Dichotomy1}
Given $\gamma>0, \eta'>0, \rho>0$,   there exist  $\beta\in (0, 1/4)$ and $\mathbf{k}\geq 1$ depending only on $n, s, \tau, \gamma, \eta', \rho, \Lambda_{0}, \Lambda_{1}, W$ such that the following holds.

Suppose $ r<1$ and $\mathbf{k}\epsilon\leq r$. For each $x_{0}\in B_{1}(0)$ satisfying $|U_{\epsilon}(x_{0}, 0)|\leq 1-\tau$, we set
\[M=\sup_{x\in B_{r}(x_{0})}\mathbf{\Theta}_{U_{\epsilon}}^{\epsilon}(2r, x).\]
If there are points $y_{0},\cdots, y_{n-1}\in B_{r}(x_{0})$ satisfying
\begin{equation}\label{rho independent}
y_{i}\not\in \mathcal{J}_{\rho r}(\langle y_{0},\cdots, y_{i-1}\rangle),\quad\text{for}~i=1,2,\cdots, n-1
\end{equation}
and
\begin{equation}\label{lem Allen-Cahn Dichotomy1.1}
\mathbf{\Theta}_{U_{\epsilon}}^{\epsilon}(2\beta r, y_{i})\geq M-\beta,\quad\text{for}~i=0,\cdots, n-1,
\end{equation}
then writing $L^{n-1}=\langle y_{0},\cdots, y_{n-1}\rangle$, we have
\begin{equation}\label{lem Allen-Cahn Dichotomy1.2}
\mathbf{\Theta}_{U_{\epsilon}}^{\epsilon}(\gamma\rho r, x)\geq M-\eta',\quad\text{on}~\mathcal{J}_{\beta r}(L^{n-1})\cap B_{r}(x_{0}).
\end{equation}
\begin{proof}
Assume that \eqref{lem Allen-Cahn Dichotomy1.2} fails. Then  there exist $\{\epsilon_{j}\}_{j\in\mathbb{N}}\subset (0, 1), \{\beta_{j}\}_{j\in\mathbb{N}},\{U_{\epsilon_{j}}\}$ satisfying \eqref{Extension equation},  $\{x_{j}\}_{j\in\mathbb{N}}\subset B_{1}(0)$, radii $\{r_{j}\}_{j\in\mathbb{N}}\subset (0, 1)$ such that $r_{j}/\epsilon_{j}\leq 2^{-j}$, $\beta_{j}\to 0$.
Moreover, there is
an collections $\{y_{0}^{(j)}, y_{1}^{(j)},\cdots, y_{n-1}^{(j)}\}\subset B_{r_{j}}(x_{j})$  such that
\begin{equation}\label{rho independentI}
y_{i}^{(j)}\not\in \mathcal{J}_{\rho r_{j}}(\langle y_{0}^{(j)},\cdots, y_{i-1}^{(j)}\rangle),\quad\text{for}~i=1,2,\cdots, n-1
\end{equation}
and
\begin{equation}\label{lem Allen-Cahn Dichotomy1.1I}
\mathbf{\Theta}_{U_{\epsilon_{j}}}^{\epsilon_{j}}(2\beta_{j}r_{j}, y_{i}^{(j)})\geq M_{j}-\beta_{j},\quad\text{for}~i=0,\cdots, n-1,
\end{equation}
but for each $j$, there exists a point $\tilde{x}_{j}\in\mathcal{J}_{\beta_{j}r_{j}}(L_{j}^{n-1})\cap B_{r_{j}}(x_{j})$ such that
\begin{equation}\label{lem3.3.1con}
\mathbf{\Theta}_{U_{\epsilon_{j}}}^{\epsilon_{j}}(\gamma\rho r_{j}, \tilde{x}_{j})<M_{j}-\eta',
\end{equation}
where $L_{j}^{n-1}=\langle y_{0}^{(j)}, y_{1}^{(j)},\cdots, y_{n-1}^{(j)} \rangle$ and
\begin{equation}\label{Mj}
M_{j}=\sup_{x\in B_{r_{j}}(x_{j})}\mathbf{\Theta}_{U_{\epsilon_{j}}}^{\epsilon_{j}}(2r_{j}, x_{j}).
\end{equation}

For each $j$, we set $\tilde{\epsilon}_{j}=\epsilon_{j}/r_{j}$ and $\tilde{U}_{\tilde{\epsilon}_{j}}(\mathbf{x})=U_{\epsilon_{j}}(\mathbf{x}_{j}+r_{j}\mathbf{x})$, where $\mathbf{x}_{j}=(x_{j}, 0)$.
 By Lemma \ref{lem apriori estimate} and Theorem \ref{convergence}, we can assume $\tilde{U}_{\tilde{\epsilon}_{j}}\to U_{\ast}$ strongly, where $U_{\ast}$ satisfies \eqref{eqn1} and \eqref{eqn2}. Moreover, we can assume
 \[M_{j}\to M_{\infty},\quad\text{as}~j\to\infty\]
 and
\[
\begin{aligned}
(y_{i}^{(j)}-x_{j})/r_{j}\to y_{i},\quad\text{for}~i=0, 1, 2, \cdots, n-1,\\
 -x_{j}/r_{j}+L_{j}^{n-1}\to L^{n-1}_{\infty}, \quad\text{for}~i=0, 1, 2, \cdots, n-1,
\end{aligned}\]
where $L_{\infty}^{n-1}=\langle y_{0}, y_{1}, \cdots, y_{n-1} \rangle.$ Since $\tilde{x}_{j}\in\mathcal{J}_{\beta_{j}r_{j}}(L_{j}^{n-1})\cap B_{r_{j}}(x_{j})$,
\[ (\tilde{x}_{j}-x_{j})/r_{j}\to x_{\infty}\in\overline{B_{1}(0)}\cap L_{\infty}^{n-1}.\]
By \eqref{Mj}, \eqref{lem3.3.1con} and the continuity of density, we have
\begin{equation}\label{lem3.3.1con1}
\mathop{\sup}\limits_{z\in B_{1}(0)}\mathbf{\Theta}_{U_{\ast}}(2,z)\leq M_{\infty}
\end{equation}
and
\[\mathbf{\Theta}_{U_{\ast}}(\gamma\rho, x_{\infty})\leq M_{\infty}-\eta'.\]
By \eqref{lem Allen-Cahn Dichotomy1.1} and the upper-semi-continuity, we get
\begin{equation}\label{lem3.3.1con2}
\mathbf{\Theta}_{U_{\ast}}(y_{i})\geq\limsup_{j\to 0} \mathbf{\Theta}_{U_{\ast}}(2\beta_{j}, y_{i})\geq M_{\infty},\quad\forall i=0,\cdots, n-1.
\end{equation}
Combining \eqref{lem3.3.1con1}, \eqref{lem3.3.1con2} with  Lemma \ref{Monotonicityformula}, we deduce that $U_{\ast}$ is homogeneous at $\mathbf{y}_{i}=(y_{i}, 0)$ in $\mathcal{B}^{+}_{2}(\mathbf{y}_{i})$ for each $i$. Since $\{y_{0}, y_{1}, \cdots, y_{n-1}\}\subset B_{1}(0)$,
there exists $\delta>0$ such that
\[\mathcal{B}^{+}_{1+\delta}(\mathbf{0})\subset\mathop{\cup}\limits_{i=0}^{k}\mathcal{B}^{+}_{2}(\mathbf{y}_{i}).\]
 Moreover, we know from Lemma \ref{Dimension reduction principle} that $U_{\ast}$ is translation invariance with respect to $L^{n-1}_{\infty}$ in $\mathcal{B}^{+}_{1+\delta}(\mathbf{0})$. In particular, we have
\[\mathbf{\Theta}_{U_{\ast}}(x_{\infty})=M_{\infty}>\mathbf{\Theta}_{U_{\ast}}(\gamma\rho; x_{\infty}),\]
contradicting the monotonicity formula.
\end{proof}
\end{lem}
\begin{lem}\label{lem Allen-Cahn Dichotomy2}
Let $\beta=\beta(n, s, \tau, \gamma, \eta', \rho, \Lambda_{0},\Lambda_{1}, W)\ll \rho$ be a constant satisfying Lemma \ref{lem Allen-Cahn Dichotomy1}. There exist positive constants $\eta\ll\beta$ and
$\mathbf{k}>1$ depending only on $n, s, \tau, \gamma, \eta', \rho, \Lambda_{0},\Lambda_{1}, W$ such that the following holds.

Suppose $r<1$ and $\mathbf{k}\epsilon\leq r$. For each $x_{0}\in B_{1}(0)$ satisfying $|U_{\epsilon}(x_{0}, 0)|\leq 1-\tau$, if there are  points $y_{0},\cdots, y_{n-1}\in B_{r}(x_{0})$ satisfying \eqref{rho independent} and
\eqref{lem Allen-Cahn Dichotomy1.1},
then by writing $L^{n-1}=\langle y_{0},\cdots, y_{n-1}\rangle$, we have
\begin{equation}\label{lem Allen-Cahn Dichotomy2.2}
\big\{|U_{\epsilon}(x, 0)|\leq 1-\tau\big\}\cap B_{r}(x_{0})\subset \mathcal{J}_{\beta r}(L^{n-1}).
\end{equation}
\end{lem}
\begin{proof}
Suppose that \eqref{lem Allen-Cahn Dichotomy2.2} fails. Then  there exist  $\{\epsilon_{j}\}_{j\in\mathbb{N}}\subset (0, 1), \{U_{\epsilon_{j}}\}$, $\{x_{j}\}_{j\in\mathbb{N}}\subset B_{1}(0)$, radii $\{r_{j}\}_{j\in\mathbb{N}}\subset (0, 1)$, $\{y_{0}^{(j)}, y_{1}^{(j)},\cdots, y_{n-1}^{(j)}\}\subset B_{r_{j}}(x_{j})$ satisfying the assumptions in the proof of Lemma \ref{lem Allen-Cahn Dichotomy1} and $\{\eta_{j}\}_{j\in\mathbb{N}}$  such that $\eta_{j}\to 0$. But for each $j$, there exists a point  $\tilde{x}_{j}\in B_{r_{j}}(x_{j})\backslash\mathcal{J}_{\beta r_{j}}(L_{j}^{n-1})$ such that \[
|U_{\epsilon_{j}}(\tilde{x}_{j}, 0)|\leq 1-\tau.\]

For each $j$, we set $\tilde{\epsilon}_{j}=\epsilon_{j}/r_{j}$ and $\tilde{U}_{\tilde{\epsilon}_{j}}(\mathbf{x})=U_{\epsilon_{j}}(\mathbf{x}_{j}+r_{j}\mathbf{x})$, where $\mathbf{x}_{j}=(x_{j}, 0)$. Then
\begin{equation}\label{xjvalue}
\Big|\tilde{U}_{\tilde{\epsilon}_{j}}(\frac{\tilde{x}_{j}-x_{j}}{r_{j}}, 0)\Big|\leq 1-\tau.
\end{equation}
By Theorem \ref{convergence}, we can assume $\tilde{U}_{\tilde{\epsilon}_{j}}\to U_{\ast}$ strongly, where $U_{\ast}$ is a stationary solution of \eqref{eqn1}. As in the proof of Lemma \ref{lem Allen-Cahn Dichotomy1}, we   get a constant $\delta>0$ such that the resulting $U_{\ast}$ is translation invariance with respect to $L_{\infty}^{n-1}$ in $\mathcal{B}_{1+\delta}^{+}(\mathbf{0})$. Moreover, for $i=0, 1,2, \cdots, n-1$, $U_{\ast}$ is homogeneous with respect to $\mathbf{y}_{i}=(y_{i}, 0)$ in $\mathcal{B}^{+}_{2}(\mathbf{y}_{i})$. Since $x_{\infty}\in\overline{B_{1}(0)}\backslash \mathcal{J}_{\beta}(L^{n-1}_{\infty})$, any blow up of $U_{\ast}$ at $\mathbf{x}_{\infty}=(x_{\infty}, 0)$ will be $n$-symmetric in $\mathcal{B}_{r_{0}}^{+}(\mathbf{x}_{\infty})$, where $r_{0}$ is a fixed positive constant. By Remark \ref{n-symmetric nonlocal minimal cone} and Lemma \ref{Clear out}, there exists a constant $s_{0}$ such that $U_{\ast}=1$ in $\mathcal{B}^{+}_{s_{0}}(\mathbf{x}_{\infty})$ or $U_{\ast}=1$ in $\mathcal{B}^{+}_{s_{0}}(\mathbf{x}_{\infty})$. Without loss of generality, we may assume $U_{\ast}=1$ in $\mathcal{B}^{+}_{s_{0}}(\mathbf{x}_{\infty})$. Now Theorem \ref{convergence} yields that $\tilde{U}_{\tilde{\epsilon}_{j}}\to 1$ uniformly on $\mathcal{B}^{+}_{s}(\mathbf{0})$ for any $s\in (0, s_{0})$, which contradicts our assumption \eqref{xjvalue}.
\end{proof}
\begin{proof}[Proof of Theorem \ref{thm Allen-Cahn Dichotomy}]
Let $\beta=\beta(n, s, \tau, \gamma, \eta', \rho, \Lambda_{0},\Lambda_{1}, W)\ll \rho$ be a constant satisfying Lemma \ref{lem Allen-Cahn Dichotomy1} and $\eta=\eta(n, s, \tau, \gamma, \eta', \rho, \Lambda_{0},\Lambda_{1}, W)\ll\beta(n, s, \tau, \gamma, \eta', \rho, \Lambda_{0},\Lambda_{1}, W)$ be a positive constant satisfying Lemma \ref{lem Allen-Cahn Dichotomy2}. Let $\mathbf{k}(n, s, \tau, \gamma, \eta', \rho, \Lambda_{0},\Lambda_{1},W)\geq 1$ be a  constant satisfying both Lemma \ref{lem Allen-Cahn Dichotomy1} and   Lemma \ref{lem Allen-Cahn Dichotomy2}.

Let $ r\in (0, 1),\mathbf{k}\epsilon\leq r, x_{0}\in B_{1}(0)$ and
\[M:=\sup_{x\in B_{2r}(x_{0})}\mathbf{\Theta}_{U_{\epsilon}}^{\epsilon}(2r, x).\]
If
\[
\{x:\mathbf{\Theta}_{U_{\epsilon}}^{\epsilon}(2\eta r, x)\geq M-\eta\}\cap B_{r}(x_{0})=\emptyset,\]
it is trivial that (ii) in Theorem \ref{thm Allen-Cahn Dichotomy} holds. Therefore, we may assume that
\[\{x:\mathbf{\Theta}_{U_{\epsilon}}^{\epsilon}(2\eta r, x)\geq M-\eta\}\cap B_{r}(x_{0})\neq \emptyset.\]
Let $y_{0}$ be a point such that $y_{0}\in \{x:\mathbf{\Theta}_{U_{\epsilon}}^{\epsilon}(2\eta r, x)\geq M-\eta\}\cap B_{r}(x_{0})$. If
\[\{x:\mathbf{\Theta}_{U_{\epsilon}}^{\epsilon}(2\eta r, x)\geq M-\eta\}\cap B_{r}(x_{0})\subset B_{\rho r}(y_{0}),\]
then (ii) in Theorem \ref{thm Allen-Cahn Dichotomy} holds. Otherwise, we can choose a point $y_{1}\not\in B_{\rho r}(y_{0})$ such that $y_{1}\in B_{r}(x_{0})$ and $\mathbf{\Theta}_{U_{\epsilon}}^{\epsilon}(2\eta r, y_{1})\geq M-\eta$.

Continuing this procedure, if we can find $y_{0},\cdots, y_{k}$ as in Lemma \ref{lem Allen-Cahn Dichotomy2}, then conclusion Theorem \ref{thm Allen-Cahn Dichotomy} (i) is immediate. Otherwise, there is some $(n-2)-$dimensional affine plane $L^{n-2}$ such that
\[\mathbf{\Theta}_{U_{\epsilon}}^{\epsilon}(2\eta r, y)<M-\eta\]
for each $y\in B_{r}(x_{0})\backslash \mathcal{J}_{\rho r}(L^{n-2})$, which is Theorem \ref{thm Allen-Cahn Dichotomy}(ii).
\end{proof}

\section{A corona-type decomposition}
\setcounter{equation}{0}
In this section, we build a crucial covering theorem using the results  from the previous sections.
\begin{thm}\label{the4.1}
There exist $\eta\in (0, 1/4), c$ and $\mathbf{k}\geq 1$ depending only on $n, s, \tau, \Lambda_{0}$, $\Lambda_{1},W$ such that the following holds.

Assume
\begin{equation}\label{B1 upperbound}
\sup\limits_{y\in B_{2}(0)}\mathbf{\Theta}_{U_{\epsilon}}^{\epsilon}(2, y)= M
\end{equation}
and $\mathbf{k}\epsilon\leq r<1$. We set
\[\mathcal{V}^{\epsilon}:=B_{1}(0)\cap \big\{|U_{\epsilon}(x, 0)|\leq 1-\tau\big\}.\]
Then there is a collection of balls $\{B_{r_{x}}(x)\}_{x\in\mathcal{U}}$, with $r_{x}\leq 1/10$ and  $x\in \mathcal{V}^{\epsilon}$, which satisfies the following properties:

(A)~Covering:
\begin{equation*}
\mathcal{V}^{\epsilon}\subset\mathop{\cup}_{x\in\mathcal{U}}B_{r_{x}}(x).
\end{equation*}

(B)~Packing:
\[\sum\limits_{x\in\mathcal{U}}r_{x}^{n-1}\leq c(n, s, \tau, \Lambda_{0}, \Lambda_{1},W).\]

(C)~Energy drop: for every $x\in\mathcal{U}$ we have $r_{x}\geq r$, and either $r_{x}=r$ or
\[\sup_{y\in B_{2r_{x}}(x)}\mathbf{\Theta}_{U_{\epsilon}}^{\epsilon}(2r_{x},y)\leq M-\frac{\eta}{2}.\]
\end{thm}

The rest of this section is devoted to the proof of Theorem \ref{the4.1}. The crucial observation that makes it works is the dichotomy Theorem \ref{thm Allen-Cahn Dichotomy}. Throughout this section, we mainly follow the arguments in \cite{Edelen-Engelstein2019}, although many suitable modifications are needed.

Let $0<\rho<1/10$ be a positive constant which will be determined later. For each integer $i\geq 0$, we write $r_{i}=\rho^{i}$.
\begin{defi}[Good/Bad ball]\label{defgoodball}
Fix a positive constant $\mathbf{k}$ satisfying all the results in the previous sections. Let $\gamma$ be a positive constant satisfying Lemma \ref{Quantutative homogeneous}. Take $x\in B_{1}(0)$ and a constant $t$ such that $\mathbf{k}\epsilon\leq r<t\leq 1$. We say a ball $B_{t}(x)$ is good if $|U_{\epsilon}(x, 0)|\leq 1-\tau$ and
\[\mathbf{\Theta}_{U_{\epsilon}}^{\epsilon}(\gamma\rho t, y)\geq M-\eta',\quad\text{on}~\big\{|U_{\epsilon}(x, 0)|\leq 1-\tau\big\}\cap B_{t}(x).\]
Otherwise, we say $B_{t}(x)$ is a bad ball.
\end{defi}
\begin{rmk}\label{new upperbound}
Using $M+\eta/2$ instead of $M$, $\rho/2$ instead of $\rho$ in Theorem \ref{thm Allen-Cahn Dichotomy}, we deduce that if $B_{t}(x)$ is a bad ball, then there exists a $(n-2)$dimensional-affine plane $L^{n-2}$ such that
\begin{equation}\label{bad ball property}
\big\{y:\mathbf{\Theta}_{U_{\epsilon}}^{\epsilon}(2\eta t, y)\geq M-\frac{\eta}{2}\big\}\cap B_{2t}(x)\subset \mathcal{J}_{\rho t}(L^{n-2}).
\end{equation}
\end{rmk}
\subsection{Good tree construction}
Suppose $B_{r_{A}}(a)$ is a good ball at scale $A\geq 0$, where $r_{A}>r$ and $a\in B_{1}(0)$. We inductively define, for each $i\geq A$, a family of good balls $\{B_{r_{i}}(g)\}_{g\in\mathcal{G}_{i}}$,   bad balls $\{B_{r_{i}}(b)\}_{b\in\mathcal{B}_{i}}$ and stop balls $\{B_{r_{i}}(d)\}_{d\in\mathcal{D}_{i}}$.

For $i=A$, we set $\mathcal{G}_{A}=\{a\}$ and $\mathcal{B}_{A}=\mathcal{D}_{A}=\emptyset$. Let $J_{A+1}$ be a maximal $2r_{A+1}/5$-net of $\mathcal{V}^{\epsilon}\cap B_{r_{A}}(a).$

If $r_{A+1}\leq r$, we set \[\mathcal{G}_{A+1}=\mathcal{B}_{A+1}=\emptyset,\quad\mathcal{D}_{A+1}=J_{A+1}.\]
In other words, we stop building the tree.

If $r_{A+1}>r$, we set
\[\mathcal{G}_{A+1}=\{y\in J_{A+1}: B_{r_{A+1}}(y)~\text{is a good ball}\},\]
\[\mathcal{B}_{A+1}=\{y\in J_{A+1}: B_{r_{A+1}}(y)~\text{is a bad ball}\}\]
and $\mathcal{D}_{A+1}=\emptyset$.

Suppose we have constructed the good/bad/stop balls down through scale $i-1$. Let $J_{i}$ be a maximal $2r_{i}/5$-net of
\[\Big(\mathcal{V}^{\epsilon}\cap B_{r_{A}}(a)\cap\big(\mathop{\cup}_{g\in\mathcal{G}_{i-1}} B_{r_{i-1}}(g)\big)\Big)\backslash\big(\mathop{\cup}\limits_{\ell=A}^{i-1}\mathop{\cup}\limits_{b\in\mathcal{B}_{l}}B_{r_{\ell}}(b)\big).\]

If $r_{i}\leq r$, we let $\mathcal{D}_{i}=J_{i}$ and $\mathcal{G}_{i}=\mathcal{B}_{i}=\emptyset$.  Otherwise, we divide the balls $\{B_{r_{i}}(y)\}_{y\in J_{i}}$ into good balls and bad balls according to definition \ref{defgoodball} and set $\mathcal{D}_{i}=\emptyset$.
 \begin{defi}\label{defgoodtree}
 The construction above is called the good tree rooted at $B_{r_{A}}(a)$, and may be written as $\mathcal{T}_{\mathcal{G}}(B_{r_{A}}(a))$. Given a good tree, we define the tree leaves
 \[\mathcal{F}(\mathcal{T}_{\mathcal{G}}(B_{r_{A}}(a))):=\{\text{the collection of bad ball centers across scales}\}\]
 and
 \[\mathcal{D}(\mathcal{T}_{\mathcal{G}}(B_{r_{A}}(a))):=\{\text{the collection of stop ball centers}\}.\]
 We use $\{r_{f}\}_{f\in \mathcal{F}(\mathcal{T}_{\mathcal{G}}(B_{r_{A}}(a)))}$  and $\{r_{d}\}_{d\in \mathcal{D}(\mathcal{T}_{\mathcal{G}}(B_{r_{A}}(a)))}$ to denote the associated radius functions for the leaves $\mathcal{F}(\mathcal{T}_{\mathcal{G}}(B_{r_{A}}(a)))$, stop ball leaves $\mathcal{D}(\mathcal{T}_{\mathcal{G}}(B_{r_{A}}(a)))$ respectively.
 \end{defi}
 \begin{lem}\label{good tree property}
Let $\mathcal{T}_{\mathcal{G}}(B_{r_{A}}(a))$ be a good tree, then
\begin{itemize}
\item[(i)]the collection of stop/ bad $r_{i}/5$- balls
\[\left\{B_{r_{b}/5}(b):b\in\mathop{\cup}\limits_{i=A}^{\infty}\mathcal{B}_{i}\right\}\bigcup \left\{B_{r_{d}/5}(d):d\in\mathop{\cup}\limits_{i=A}^{\infty}\mathcal{D}_{i}\right\}\]
are pairwise disjoint;
\item[(ii)] for each $y\in\mathcal{F}(\mathcal{T}_{\mathcal{G}}(B_{r_{A}}(a)))\cup \mathcal{D}(\mathcal{T}_{\mathcal{G}}(B_{r_{A}}(a)))$, we have
\begin{equation}\label{small energy drop}
\mathbf{\Theta}_{U_{\epsilon}}^{\epsilon}(\gamma r_{y}, y)\geq M-\eta'.
\end{equation}
\end{itemize}
 \end{lem}
 \begin{proof}
 Let $y, y'$ be two points in $\mathcal{F}(\mathcal{T}_{\mathcal{G}}(B_{r_{A}}(a)))\cup \mathcal{D}(\mathcal{T}_{\mathcal{G}}(B_{r_{A}}(a)))$. First, suppose $y\in\mathcal{B}_{i}, y' \in\mathcal{B}_{i'}$.

 If $i'=i$, we get from the construction and the definition of $2r_{i}/5-$ net that $|y-y'|\geq 2r_{i}/5$. Therefore,
 $B_{r_{i}/5}(y)\cap B_{r_{i}/5}(y')=\emptyset.$

 If $i'>i$, we know from the construction that \[y'\not\in\mathop{\cup}\limits_{\ell=A}^{i'-1}\mathop{\cup}\limits_{b\in\mathcal{B}_{\ell}}B_{r_{\ell}}(b).\]
 In particular, we have $|y-y'|\geq r_{i}$. Since $r_{i}/5+r_{i'}/5<r_{i}$,
$B_{r_{i}/5}(y)\cap B_{r_{i}/5}(y')=\emptyset.$
 The other cases can be discussed similarly, hence we have proved (i).

If $y\in\mathcal{B}_{i}\subset \mathcal{F}(\mathcal{T}_{\mathcal{G}}(B_{r_{A}}(a)))$, then there exist an integer $i\geq A+1$  and a point $g\in\mathcal{G}_{i-1}$ such that $y\in B_{r_{i-1}}(g)$. Since $B_{r_{i-1}}(g)$ is a good ball, then we get from Definition \ref{defgoodball} that
\[
\mathbf{\Theta}_{U_{\epsilon}}^{\epsilon}(\gamma\rho r_{i-1}, y)=\mathbf{\Theta}_{U_{\epsilon}}^{\epsilon}(\gamma r_{i}, y)\geq M-\eta'. \]
Because $r_{y}=r_{i}$,  we  get \eqref{small energy drop}. Similarly, we can prove that if $y\in\mathcal{D}(\mathcal{T}_{\mathcal{G}}(B_{r_{A}}(a)))$, then \eqref{small energy drop} holds.
 \end{proof}
 \begin{thm}\label{thegoodtree}
Let $\mathcal{T}_{\mathcal{G}}(B_{r_{A}}(a))$ be a good tree. We have

(A) Tree-leaf packing:
\[\sum\limits_{f\in \mathcal{F}(\mathcal{T}_{\mathcal{G}}(B_{r_{A}}(a)))}r_{f}^{n-1}\leq c_{1}(n) r_{A}^{n-1}.\]

(B) Stop ball packing:
\[\sum\limits_{d\in \mathcal{D}(\mathcal{T}_{\mathcal{G}}(B_{r_{A}}(a)))}r_{d}^{n-1}\leq c(n)r_{A}^{n-1}.\]

(C) Covering:
\[\mathcal{V}^{\epsilon}\cap B_{r_{A}}(a)\subset\big(\mathop{\cup}\limits_{d\in \mathcal{D}(\mathcal{T}_{\mathcal{G}}(B_{r_{A}}(a)))}B_{r_{d}}(d)\big)\bigcup\big(\mathop{\cup}\limits_{f\in \mathcal{F}(\mathcal{T}_{\mathcal{G}}(B_{r_{A}}(a)))}B_{r_{f}}(f)\big).\]

(D) Stop ball structure: for any $d\in \mathcal{D}(\mathcal{T}_{\mathcal{G}}(B_{r_{A}}(a)))$, we have $\rho r<r_{d}\leq r.$
\end{thm}
\begin{proof}
Since $a\in B_{1}(0)$ and $r_{A}\leq 1$, we know from the monotonicity formula that
\[\sup_{y\in B_{r_{A}}(a)}\mathbf{\Theta}_{U_{\epsilon}}^{\epsilon}(2r_{a}, y)\leq \sup_{y\in B_{2}(0)}\mathbf{\Theta}_{U_{\epsilon}}^{\epsilon}(2, y)\leq M.\]
Since the collections of stop/ bad $r_{i}/5$-balls
\[\left\{B_{r_{i}/5}(b):b\in\mathop{\cup}\limits_{i=A}^{\infty}\mathcal{B}_{i}\right\}\bigcup \left\{B_{r_{i}/5}(d):d\in\mathop{\cup}\limits_{i=A}^{\infty}\mathcal{D}_{i}\right\}\]
are  pairwise disjoint and centered in $\mathcal{V}^{\epsilon}$, we can apply Lemma \ref{key packing} at scale $ B_{r_{A}}(a)$ to obtain (A) and (B).

 Conclusion (C) follows from an elementary induction argument. Indeed, for each $i\geq A$, we claim that
\begin{equation}\label{thegoodtree1.1}
\mathcal{V}^{\epsilon}\cap B_{r_{A}}(a)\subset \big(\mathop{\cup}\limits_{g\in\mathcal{G}_{i}}B_{r_{i}}(g)\big)\bigcup\big(\mathop{\cup}\limits_{l=A}^{i}\mathop{\cup}\limits_{y\in\mathcal{B}_{l}\cup\mathcal{D}_{l}}B_{r_{l}}(y)\big).
\end{equation}
When $i=A$, \eqref{thegoodtree1.1} trivially holds. Assume that \eqref{thegoodtree1.1} holds for $i-1$.  By the construction of good tree, we have
\begin{equation}\label{thegoodtree1.2}
\Big(\mathcal{V}^{\epsilon}\cap B_{r_{A}} (a)\cap \big(\mathop{\cup}\limits_{g\in\mathcal{G}_{i-1}}B_{r_{i-1}}(g)\big)\Big)\backslash\big(\mathop{\cup}\limits_{l=A}^{i-1}\mathop{\cup}\limits_{b\in\mathcal{B}_{l}}B_{r_{l}}(b)\big)\subset \big(\mathop{\cup}\limits_{y\in \mathcal{B}_{i}\cup \mathcal{G}_{i}\cup \mathcal{D}_{i}} B_{r_{i}}(y)\big).
\end{equation}
Hence
\[\begin{aligned}
\mathcal{V}^{\epsilon}\cap B_{r_{A}}(a)\subset& \Big(\mathcal{V}^{\epsilon}\cap B_{r_{A}}(a)\cap\big(\mathop{\cup}\limits_{g\in\mathcal{G}_{i-1}}B_{r_{i-1}}(g)\big)\Big)\bigcup\big(\mathop{\cup}\limits_{l=A}^{i-1}\mathop{\cup}\limits_{y\in\mathcal{B}_{l}\cup\mathcal{D}_{l}}B_{r_{l}}(y)\big)\\
\subset&\big(\mathop{\cup}\limits_{l=A}^{i-1}\mathop{\cup}\limits_{y\in\mathcal{B}_{l}\cup\mathcal{D}_{l}}B_{r_{l}}(y)\big)\mathop{\bigcup}\big(\mathop{\cup}\limits_{y\in \mathcal{G}_{i}\cup\mathcal{B}_{i}\cup\mathcal{D}_{i}}B_{r_{i}}(y)\big)\\
\subset&\big(\mathop{\cup}\limits_{g\in\mathcal{G}_{i}}B_{r_{i}}(g)\big)\bigcup\big(\mathop{\cup}\limits_{l=A}^{i}\mathop{\cup}\limits_{y\in\mathcal{B}_{l}\cup\mathcal{D}_{l}}B_{r_{l}}(y)\big).
\end{aligned}\]
This proves \eqref{thegoodtree1.1} at stage $i$. When $r_{i}\leq r$, we have $\mathcal{G}_{i}=\mathcal{B}_{i}=\emptyset$, so \eqref{thegoodtree1.1} implies
\[\mathcal{V}^{\epsilon}\cap B_{r_{A}}(a)\subset \mathop{\cup}\limits_{l=A}^{i}\mathop{\cup}\limits_{y\in\mathcal{B}_{l}\cup\mathcal{D}_{l}}B_{r_{l}}(y).\]
This is exactly (C).

Notice that $\mathcal{D}_{i}\neq \emptyset$ only if $r_{i}\leq r$, in which case $r_{i-1}>r$. Then conclusion (D) follows immediately from the definition of $r_{i}$.
\end{proof}

\subsection{Bad tree construction}
Suppose $B_{r_{A}}(a)$ is a bad ball at scale $A\geq 0$, where $r_{A}>r$ and $a\in B_{1}(0)$. We define the bad tree rooted at $B_{r_{A}}(a)$. This is almost the same with the construction of good trees.

We inductively define, for each $i\geq A$,  a family of good balls $\{B_{r_{i}}(g)\}_{g\in\mathcal{G}_{i}}$, bad balls $\{B_{r_{i}}(b)\}_{b\in\mathcal{B}_{i}}$, and stop balls $\{B_{r_{i}}(d)\}_{d\in\mathcal{D}_{i}}$.

 For $i=A$, we let $\mathcal{B}_{A}=\{a\}$ and $\mathcal{G}_{A}=\mathcal{S}_{A}=\emptyset$.

 Since $B_{r_{A}}(a)$ is a bad ball, we have a $(n-2)-$dimensional affine plane $L_{a}^{n-2}$ such that
\[\big\{y:\mathbf{\Theta}_{U_{\epsilon}}^{\epsilon}(2\eta r_{A}, y)\geq M-\frac{\eta}{2}\big\}\cap B_{2r_{A}}(a)\subset\mathcal{J}_{\rho r_{A}}(L_{a}^{n-2}).\]

If $r_{A+1}\leq r$, then take $\mathcal{G}_{A+1}=\mathcal{B}_{A+1}=\emptyset$, and $\mathcal{D}_{A+1}$ to be a maximal $2\eta r_{A}/5$-net in
$\mathcal{V}^{\epsilon}\cap B_{r_{A}}(a).$
 That is to say, we stop the construction of the tree.
 Remember that $\eta\ll\rho$ according to Lemma \ref{lem Allen-Cahn Dichotomy2}.

 If $r_{A+1}>r$, we define $\mathcal{D}_{A+1}$ to be a maximal $2\eta r_{A}/5$-net in
\[\big(\mathcal{V}^{\epsilon}\cap B_{r_{A}}(a)\big)\backslash \mathcal{J}_{2\rho r_{A}}(L_{a}^{n-2})\]
and we let $\{g:g\in\mathcal{G}_{A+1}\}\cup\{b:b\in\mathcal{B}_{A+1}\}$ be a maximal $2r_{A+1}/5$-net in
\[\big(\mathcal{V}^{\epsilon}\cap B_{r_{A}}(a)\big)\cap \mathcal{J}_{2\rho r_{A}}(L_{a}^{n-2}).\]

Suppose we have constructed the good/bad/stop balls down through scale $i-1$. For each $b\in \mathcal{B}_{i-1}$, we have a $(n-2)$-dimensional affine plane $L_{b}^{n-2}$ such that
\[\big\{y:\mathbf{\Theta}_{U_{\epsilon}}^{\epsilon}(2\eta r_{i-1}, y)\geq M-\frac{\eta}{2}\big\}\cap B_{2r_{i-1}}(b)\subset\mathcal{J}_{\rho r_{i-1}}(L_{b}^{n-2}).\]

If $r_{i}\leq r$, then take $\mathcal{G}_{i}=\mathcal{B}_{i}=\emptyset$, and $\mathcal{D}_{i}$ to be a maximal $2\eta r_{i-1}/5$-net in
\[\mathcal{V}^{\epsilon}\cap B_{r_{A}}(a)\cap \big(\mathop{\cup}\limits_{b\in \mathcal{B}_{i-1}}B_{r_{i}-1}(b)\big).\]

 If $r_{i}>r$, we define $\mathcal{D}_{i}$ to be a maximal $2\eta r_{i-1}/5$-net in
\[\mathcal{V}^{\epsilon}\cap B_{r_{A}}(a)\cap\Big(\bigcup\limits_{b\in\mathcal{B}_{i-1}}\big(B_{r_{i-1}}(b)\backslash\mathcal{J}_{2\rho r_{i-1}}(L_{b}^{n-2})\big)\Big)\]
and we let $\{g:g\in\mathcal{G}_{i}\}\cup\{b:b\in\mathcal{B}_{i}\}$ be a maximal $2r_{i}/5$-net in
\[\mathcal{V}^{\epsilon}\cap B_{r_{A}}(a)\cap\Big(\bigcup\limits_{b\in\mathcal{B}_{i-1}}\big(B_{r_{i-1}}(b)\cap\mathcal{J}_{2\rho r_{i-1}}(L_{b}^{n-2})\big)\Big).\]
This completes the construction of the bad tree.
\begin{defi}\label{defbadtree}
The construction above is called the bad tree rooted at $B_{r_{A}}(a)$, and may be written as $\mathcal{T}_{\mathcal{B}}(B_{r_{A}}(a))$. Given  a bad tree, we define the tree leaves
\[\mathcal{F}(\mathcal{T}_{\mathcal{B}}(B_{r_{A}}(a))):=\{\text{the collections of good ball centers}\}\] and
\[\mathcal{D}(\mathcal{T}_{\mathcal{B}}(B_{r_{A}}(a))):=\{\text{the collection of stop ball centers}\}.\]
We write $r_{f}, r_{d}$ for the associated radius function respectively.
\end{defi}
\begin{rmk} It is easy to see that if $d\in\mathcal{D}_{i}\subset \mathcal{D}(\mathcal{T}_{\mathcal{B}}(B_{r_{A}}(a)))$, then $r_{d}=\eta r_{i-1}$.
\end{rmk}
 \begin{thm}\label{thebadtree}
Let $\mathcal{T}_{\mathcal{B}}(B_{r_{A}}(a))$ be a bad tree. We have

(A) Tree-leaf packing:
\[\sum\limits_{f\in \mathcal{F}(\mathcal{T}_{\mathcal{B}}(B_{r_{A}}(a)))}r_{f}^{n-1}\leq 2c_{2}(n)\rho r_{A}^{n-1}.\]

(B) Stop ball packing:
\[\sum\limits_{d\in \mathcal{D}(\mathcal{T}_{\mathcal{B}}(B_{r_{A}}(a)))}r_{d}^{n-1}\leq c(n, \eta)r_{A}^{n-1}.\]

(C) Covering control:
\[\mathcal{V}^{\epsilon}\cap B_{r_{A}}(a)\subset\big(\mathop{\cup}\limits_{d\in \mathcal{D}(\mathcal{T}_{\mathcal{B}}(B_{r_{A}}(a)))}B_{r_{d}}(d)\big)\bigcup\big(\mathop{\cup}\limits_{f\in \mathcal{F}(\mathcal{T}_{\mathcal{B}}(B_{r_{A}}(a)))}B_{r_{f}}(f)\big).\]

(D) Stop ball structure: for any $d\in \mathcal{D}(\mathcal{T}_{\mathcal{B}}(B_{r_{A}}(a)))$, we have $\eta r< r_{d}<r$ or
\[\sup_{y\in B_{2r_{d}}(d)}\mathbf{\Theta}_{U_{\epsilon}}^{\epsilon}(2r_{d}, y)\leq M-\frac{\eta}{2}.\]
\end{thm}
\begin{proof}
Take $r_{i}>r$ and $b\in\mathcal{B}_{i-1}$. The good/ bad ball centers
\[(\mathcal{G}_{i}\cup\mathcal{B}_{i})\cap B_{r_{i-1}}(b)\subset \mathcal{J}_{2\rho r_{i-1}}(L^{n-2}_{b})\]
and the $r_{i}/5$-balls are disjoint. Therefore,
\[\#\{(\mathcal{G}_{i}\cup\mathcal{B}_{i})\cap B_{r_{i-1}}(b)\}\leq \frac{\omega_{n-2}\omega_{2}(3\rho)^{2}}{\omega_{n}(\rho/5)^{n}}\leq c_{2}(n)\rho^{2-n},\]
where $\#\{(\mathcal{G}_{i}\cup\mathcal{B}_{i})\cap B_{r_{i-1}}(b)\}$ is the number of points in $(\mathcal{G}_{i}\cup\mathcal{B}_{i})\cap B_{r_{i-1}}(b)$. We deduce
\[\#\{\mathcal{G}_{i}\cup\mathcal{B}_{i}\}r_{i}^{n-2}\leq c_{2}\rho\#\{\mathcal{B}_{i-1}\}r_{i-1}^{n-2}\leq c_{2}\rho \#\{\mathcal{G}_{i-1}\cup\mathcal{B}_{i-1}\}r_{i-1}^{n-2}\leq\cdots\leq(c_{2}\rho)^{i-A}r_{A}^{n-2}.\]
Up to now, the constant $\rho$ has not been determined. At this stage, we fix a positive constant $\rho$ such that  $c_{2}\rho<1/2$. Then
\begin{equation}\label{good leaves packing}
\sum\limits_{i=A+1}^{\infty}\#\{\mathcal{G}_{i}\cup\mathcal{B}_{i}\}r_{i}^{n-2}\leq\sum\limits_{i=A+1}^{\infty}(c_{2}\rho)^{i-A}r_{A}^{n-2}\leq 2c_{2}\rho r_{A}^{n-2}.
\end{equation}
For each $f\in\mathcal{F}(\mathcal{T}_{\mathcal{B}}(B_{r_{A}}(a)))$, we have $r<f_{f}\leq r_{A+1}$. Thus (A) follows from \eqref{good leaves packing}.

Given $i\geq A+1$, the  balls $\{B_{\frac{r_{d}}{5}}(d)\}_{d\in\mathcal{D}_{i}}$ are disjoint.  Since each $d$ belongs to a bad ball, this implies that
\[\#\{\mathcal{D}_{i}\}\leq\frac{10^{n}}{\eta^{n}}\#\{\mathcal{B}_{i-1}\}.\]
Since there is no stop ball at scale $r_{A}$, we deduce
\[\sum\limits_{i=A+1}^{\infty}\#\{\mathcal{D}_{i}\}(\eta r_{i-1})^{n-1}\leq 10^{n}\eta^{-1}\sum\limits_{i=A}^{\infty}\#\{B_{i}\}r_{i}^{n-1}\leq c(n, \eta)r_{A}^{n-1}.\]
This proves estimate (B).

 For each $i\geq A$, we claim that
\begin{equation}\label{thebadtree1.1}
\mathcal{V}^{\epsilon}\cap B_{r_{A}}(a)\subset \big(\mathop{\cup}\limits_{b\in\mathcal{B}_{i}}B_{r_{i}}(b)\big)\cup\big(\mathop{\cup}\limits_{l=A}^{i}\mathop{\cup}\limits_{y\in \mathcal{G}_{l}\cup\mathcal{D}_{l}}B_{r_{y}}(y)\big).
\end{equation}
When $i=A$, \eqref {thebadtree1.1} holds trivially. Assume  \eqref{thebadtree1.1} holds for $i-1$. Then by construction,
\begin{equation}\label{thebadtree1.2}
\mathcal{V}^{\epsilon}\cap B_{r_{A}} (a)\cap \big(\mathop{\cup}\limits_{b\in\mathcal{B}_{i-1}}B_{r_{i-1}}(b)\big)\subset \big(\mathop{\cup}\limits_{y\in \mathcal{B}_{i}\cup \mathcal{G}_{i}\cup \mathcal{D}_{i}} B_{r_{y}}(y)\big).
\end{equation}
Thus
\[\begin{aligned}
\mathcal{V}^{\epsilon}\cap B_{r_{A}}(a)\subset& \Big(\mathcal{V}^{\epsilon}\cap B_{r_{A}}(a)\cap\big(\mathop{\cup}\limits_{b\in\mathcal{B}_{i-1}}B_{r_{i-1}}(b)\big)\Big)\bigcup\big(\mathop{\cup}\limits_{l=A}^{i-1}\mathop{\cup}\limits_{y\in\mathcal{G}_{l}\cup\mathcal{D}_{l}}B_{r_{y}}(y)\big)\\
\subset&\big(\mathop{\cup}\limits_{l=A}^{i-1}\mathop{\cup}\limits_{y\in\mathcal{G}_{l}\cup\mathcal{D}_{l}}B_{r_{y}}(y)\big)\mathop{\bigcup}\big(\mathop{\cup}\limits_{y\in \mathcal{G}_{i}\cup\mathcal{B}_{i}\cup\mathcal{D}_{i}}B_{r_{i}}(y)\big)\\
\subset&\big(\mathop{\cup}\limits_{b\in\mathcal{B}_{i}}B_{r_{i}}(b)\big)\bigcup\big(\mathop{\cup}\limits_{l=A}^{i}\mathop{\cup}\limits_{y\in\mathcal{G}_{l}\cup\mathcal{D}_{l}}B_{r_{y}}(y)\big).
\end{aligned}\]
This proves \eqref{thebadtree1.1} at stage $i$. When $r_{i}\leq r$, there is no  bad ball, so  (C) follows from \eqref{thebadtree1.1}.

Finally, we prove conclusion (D). Take a stop ball center $d\in\mathcal{D}_{i}$. If $r_{d}>r$, then necessarily $d\in B_{r_{i-1}}(b)\backslash B_{2\rho r_{i-1}}(L_{b}^{n-2})$ for some bad center $b\in \mathcal{B}_{i-1}$. By  Definition \ref{defgoodball} and the choice that $\eta<\rho/2$, we have
\[\sup\limits_{y\in B_{2r_{d}}(d)}\mathbf{\Theta}_{U_{\epsilon}}^{\epsilon}(2r_{d},y)\leq \sup\limits_{y\in B_{\rho r_{i-1}}(d)}\mathbf{\Theta}_{U_{\epsilon}}^{\epsilon}(2\eta r_{i-1}, y)\leq M-\frac{\eta}{2}.\]
Conversely, the only way $r_{i}\leq r$ occur is that $r_{i-1}> r$. In this case, we have
\[r\geq  \rho r_{i-1}> \eta r_{i-1}=r_{d}> \eta r. \qedhere\]
\end{proof}
\subsection{Alternating the trees}
In this subsection, we fix a positive constant $\rho$ such that
\[c_{1}(n)c_{2}(n)\rho\leq \frac{1}{4},\]
where $c_{1}(n)$ is the positive constant in Theorem \ref{thegoodtree} and  $c_{2}(n)$ is the positive constant in Theorem \ref{thebadtree}.

We inductively define for each $i=0, 1,2,\cdots$ a family of tree leaves $\{B_{r_{f}}(f)\}_{f\in\mathcal{F}_{i}}$ and stop balls $\{B_{r_{d}}(d)\}_{d\in D_{i}}$. Here $r_{f}, r_{d}$ are the associated radius functions.

Let $\mathcal{F}_{0}=\{0\}$, and define the associated radius function $r_{f=0}=1$, so the ball $B_{1}(0)$ is the only leaf at stage $0$. We let $\mathcal{D}_{0}=\emptyset$, so there are no stop balls at stage $0$. Trivially, the leaf $\mathcal{F}_{0}$ is either a good ball or a bad ball.

Suppose we have defined the  leaves and stop balls up to stage $i-1$. The leaves in $\mathcal{F}_{i-1}$ are all good ball centers or bad ball centers.  If they are good, let us define for each $f\in\mathcal{F}_{i-1}$ a good tree $\mathcal{T}_{\mathcal{G}, f}=\mathcal{T}_{\mathcal{G}}(B_{r_{f}}(f))$, with parameters $\rho$ and $\eta$ defined above. Then we set
\[\mathcal{F}_{i}=\mathop{\cup}\limits_{f\in\mathcal{F}_{i-1}}\mathcal{F}(\mathcal{T}_{\mathcal{G}, f})\]
and
\[\mathcal{D}_{i}=\mathcal{D}_{i-1}\cup\mathop{\cup} \limits_{f\in\mathcal{D}_{i-1}}\mathcal{D}(\mathcal{T}_{\mathcal{G}, f}).\]
By the construction, the leaves $\mathcal{F}_{i}$ are bad ball centers.

On the other hand, if all the leaves $\mathcal{F}_{i-1}$ are bad ball centers, let us define for each $f\in\mathcal{F}_{i-1}$ a bad tree $\mathcal{T}_{\mathcal{B}, f}=\mathcal{T}_{\mathcal{B}}(B_{r_{f}}(f))$. Then we set
\[\mathcal{F}_{i}=\mathop{\cup}\limits_{f\in\mathcal{F}_{i-1}}\mathcal{F}(\mathcal{T}_{\mathcal{B}, f})\]
and
\[\mathcal{D}_{i}=\mathcal{D}_{i-1}\cup\mathop{\cup} \limits_{f\in\mathcal{D}_{i-1}}\mathcal{D}(\mathcal{T}_{\mathcal{B}, f}).\]
In this case, all the leaves $\mathcal{F}_{i}$ are good ball centers.
 \begin{thm}\label{thealttree}
There is some integer $N$ such that $\mathcal{F}_{N}=\emptyset$, and we have:

(A) Tree-leaf packing:
\[\sum\limits_{i=0}^{N-1}\sum\limits_{f\in \mathcal{F}_{i}}r_{f}^{n-1}\leq c(n).\]

(B) Stop ball packing:
\[\sum\limits_{d\in \mathcal{D}_{N}}r_{d}^{n-1}\leq c(n, s, \tau, \Lambda_{0}, \Lambda_{1}, W).\]

(C) Covering control:
\[\mathcal{V}^{\epsilon}\subset\mathop{\cup}\limits_{d\in \mathcal{D}_{N}} B_{r_{d}}(d).\]

(D) Stop ball structure: for any $d\in \mathcal{D}_{N}$, we have $\eta r\leq r_{d}\leq r$ or
\[\sup_{y\in B_{2r_{d}}(d)}\mathbf{\Theta}_{U_{\epsilon}}^{\epsilon}(2r_{d},y)\leq M-\frac{\eta}{2}.\]
\end{thm}
\begin{proof}
From the construction, in any given good tree or bad tree, we have
\[\max\limits_{f\in\mathcal{F}_{i}}r_{f}\leq \rho\max\limits_{f\in\mathcal{F}_{i-1}}r_{f}\leq \rho^{i}.\]
For $i$ sufficiently large, we would have $\rho^{i}<r$. It follows easily that $\mathcal{F}_{N}=\emptyset$ provided  $N$ is large enough.

Let us prove (A). Suppose $\mathcal{F}_{i}$ are good ball centers. The collection of centers $\{f\in\mathcal{F}_{i}\}$ are precisely the leaves of bad trees rooted at $\{f'\in\mathcal{F}_{i-1}\}$. Therefore, using Theorem \ref{thebadtree}, we have
\[\sum\limits_{f\in\mathcal{F}_{i}}r_{f}^{n-1}\leq 2c_{2}\rho \sum\limits_{f\in\mathcal{F}_{i-1}}r_{f}^{n-1}.\]
Conversely, if  $\mathcal{F}_{i}$ are bad ball centers, then the points in $\mathcal{F}_{i}$ are all leaves of good trees rooted at $\mathcal{F}_{i-1}$. By Theorem \ref{thegoodtree},
\[\sum\limits_{f\in\mathcal{F}_{i}}r_{f}^{n-1}\leq c_{1} \sum\limits_{f\in\mathcal{F}_{i-1}}r_{f}^{n-1}.\]
 We deduce that
\[\sum\limits_{f\in\mathcal{F}_{i}}r_{f}^{n-1}\leq c(n)(2c_{1}c_{2}\rho)^{\frac{i}{2}}\leq c(n)2^{-\frac{i}{2}}.\]
The packing estimate (A) follows directly.

We prove (B). Each stop ball $d\in D_{N}$ arises from  a good or bad tree rooted in some $f\in\mathcal{F}_{i}$, for some $i<N$. Using Theorem \ref{thegoodtree}, Theorem \ref{thebadtree} and (A), we have
\[\sum\limits_{d\in\mathcal{D}_{N}}r_{d}^{n-1}\leq c(n, s, \tau, \Lambda_{0}, \Lambda_{1}, W)\sum\limits_{i=0}^{N}\sum\limits_{f\in\mathcal{F}_{i}}r_{f}^{n-1}\leq c(n, s, \tau, \Lambda_{0}, \Lambda_{1}, W).\]
Applying Theorem \ref{thegoodtree} and Theorem \ref{thebadtree} to each tree constructed at $f\in\mathcal{F}_{i-1}$, we get
\[\mathop{\cup}\limits_{f\in\mathcal{F}_{i-1}}\big(\mathcal{V}^{\epsilon}\cap B_{r_{f}}(f)\big)\subset\mathop{\cup}\limits_{f\in\mathcal{F}_{i}}\big(\mathcal{V}^{\epsilon}\cap B_{r_{f}}(f)\big)\cup\mathop{\cup}_{d\in\mathcal{D}_{i}}B_{r_{d}}(d).\]
Since $\mathcal{V}^{\epsilon}\subset B_{1}(0)$, it follows by induction that
\[\mathcal{V}^{\epsilon}\subset\big(\mathop{\cup}\limits_{f\in\mathcal{F}_{i}}B_{r_{f}}(f)\big)\cup \big(\mathop{\cup}\limits_{d\in\mathcal{D}_{i}}B_{r_{d}}(d)\big)\]
for each $i$. Therefore, we can obtain (C) by setting $i=N$.

Finally, (D) is a  direct consequence of Theorem \ref{thegoodtree} and Theorem \ref{thebadtree}.
\end{proof}

With the help of the above construction, now we can prove Theorem \ref{the4.1}.
\begin{proof}[Proof of Theorem \ref{the4.1}]
 Let $\mathcal{U}=\mathcal{D}_{N}$, and given $x=d\in\mathcal{U}$, define the radius function
 \[r_{x}=\max\{r, r_{d}\}.\]
 By Theorem \ref{thealttree} (C), it is easy to see that
 \[\mathcal{V}^{\epsilon}\subset\mathop{\cup}_{x\in\mathcal{U}}B_{r_{x}}(x).\]
 By Theorem \ref{thealttree} (B), (D) and the definition of $r_{x}$, we have
 \[\sum_{x\in\mathcal{U}}r_{x}^{n-1}=\sum_{\{d\in\mathcal{U}:r_{d}\leq r\}}r^{n-1}+\sum_{\{d\in\mathcal{U}:r_{d}> r\}}r_{d}^{n-1}\leq c(n, s, \tau,\Lambda_{0}, \Lambda_{1}).\]
 Using Theorem \ref{thealttree} once again, we know that if $x\in\mathcal{U}$ and $r_{x}>r$, then
 \[\sup_{y\in B_{2r_{x}}(x)}\mathbf{\Theta}_{U_{\epsilon}}^{\epsilon}(2r_{x},y)\leq M-\frac{\eta}{2}.\]
 Therefore,  $\{B_{r_{x}}(x)\}_{x\in \mathcal{U}}$ is a cover of $\mathcal{V}^{\epsilon}$ satisfying Theorem \ref{the4.1}.
\end{proof}

\section{The proof of of Theorem \ref{themain1}}
\setcounter{equation}{0}
In this section, we use Theorem \ref{the4.1} to prove Theorem \ref{themain1}.
\begin{proof}[Proof of Theorem \ref{themain1}]
Fix two constants
$\eta, \mathbf{k}$  so that they satisfy Theorem \ref{the4.1}.

If $r\geq \max\{10^{-[(10^{n}\Lambda_{1})/\eta]-1}, \mathbf{k}\epsilon\}$,
where $[(10^{n}\Lambda_{1})/\eta]$ is the integer part of $(10^{n}\Lambda_{1})/\eta$, then we can always find a cover of $B_{1}(0)$ by balls of radius $r$. By choosing a subcovering if necessary, we can make the number of the balls to be controlled by a universal constant. In this case,  Theorem \ref{themain1} holds trivially.

If $r\in (\mathbf{k}\epsilon, 10^{-[(10^{n}\Lambda_{1})/\eta]-1})$, we can use Theorem \ref{the4.1} to build a cover $\mathcal{U}_{1}$ of $\big\{|U_{\epsilon}(x, 0)|\leq 1-\tau\big\}\cap B_{1}(0)$. There are two cases we need to discuss separately.

\textbf{Case 1.} For each $x\in\mathcal{U}_{1}$, we have $r_{x}=r$.

The packing and covering estimates of Theorem \ref{the4.1} (A), (B) imply
\[\big\{|U_{\epsilon}(x, 0)|\leq 1-\tau\big\}\cap B_{1}(0)\subset \mathop{\cup}\limits_{x\in\mathcal{U}_{1}}B_{r}(x)\]
and
\[\#\{x\in \mathcal{U}_{1}\}r^{n-1}\leq c(n, s, \tau, \Lambda_{0}, \Lambda_{1},W).\]
Then
 \begin{equation}\label{prothemain1.1}
 \mathscr{L}^{n}\Big(\mathcal{J}_{r}\big(\{|U_{\epsilon}(x, 0)|\leq 1-\tau\}\cap B_{1}(0)\big)\Big)\leq\#\{x\in \mathcal{U}_{1}\}\mathscr{L}^{n}(B_{2r}(0))\leq cr.
 \end{equation}
 Hence the proof of Theorem \ref{themain1} is finished.

\textbf{Case 2.} There exists some $x\in\mathcal{U}_{1}$ such that  $r_{x}>r$.

We will use Theorem \ref{the4.1} inductively to build a (finite) sequence of refinements $\mathcal{U}_{1},\mathcal{U}_{2},\mathcal{U}_{3},\cdots$, which satisfy for each $i$ the following properties:

($\text{A}_{i}$)~Covering:
\begin{equation*}
\big\{|U_{\epsilon}(x, 0)|\leq 1-\tau\big\}\cap B_{1}(0)\subset\mathop{\cup}_{x\in\mathcal{U}_{i}}B_{r_{x}}(x).
\end{equation*}

($\text{B}_{i}$)~Packing
\[\sum\limits_{x\in\mathcal{U}_{i}}r_{x}^{n-1}\leq c(1+\sum\limits_{x\in\mathcal{U}_{i-1}}r_{x}^{n-1}).\]

($\text{C}_{i}$)~Energy drop: for every $x\in\mathcal{U}_{i}$ we have $r_{x}\geq r$, and either $r_{x}=r$ or
\[\sup_{y\in B_{2r_{x}}(x)}\mathbf{\Theta}_{U_{\epsilon}}^{\epsilon}(2r_{x}, y)\leq M-\frac{i\eta}{2}.\]

($\text{D}_{i}$)~Radius control: we have
\[\sup\limits_{x\in\mathcal{U}_{i}}r_{x}\leq 10^{-i}\quad\text{and}\quad \mathcal{U}_{i}\subset B_{1+10^{-i}}(0)\cap \big\{|U_{\epsilon}(x, 0)|\leq 1-\tau\big\}.\]

Notice that $\mathbf{\Theta}_{U_{\epsilon}}^{\epsilon}(\cdot, \cdot)$ is nonnegative. If $i\geq (2M/\eta)+1$, then each $x\in \mathcal{U}_{i}$ will necessarily satisfy
$r_{x}=r$. Arguing as in Case 1, we can obtain a bound like
\[\mathscr{L}^{n}\Big(\mathcal{J}_{r}\big(\{|U_{\epsilon}(x, 0)|\leq 1-\tau\}\cap B_{1}(0)\big)\Big)\leq cr.\]
Hence $\mathcal{U}_{i}$ is exactly the cover we need.

Therefore, it remains to build  a (finite) sequence of refinements $\mathcal{U}_{1},\mathcal{U}_{2},\mathcal{U}_{3},\cdots$
such that for each $i$, $\mathcal{U}_{i}$ satisfies ($\text{A}_{i}$), ($\text{B}_{i}$), ($\text{C}_{i}$),($\text{D}_{i}$).
We have already constructed $\mathcal{U}_{1}$ which satisfies ($\text{A}_{1}$), ($\text{B}_{1}$), ($\text{C}_{1}$), ($\text{D}_{1}$) from Theorem \ref{the4.1}. Suppose we have constructed $\mathcal{U}_{i-1}$ satisfying ($\text{A}_{i-1}$), ($\text{B}_{i-1}$), ($\text{C}_{i-1}$), ($\text{D}_{i-1}$). For each $x\in\mathcal{U}_{i-1}$ with $r_{x}>r$, we wish to apply Theorem \ref{the4.1} at scale $B_{r_{x}}(x)$ to obtain a new collection $\mathcal{U}_{i, x}$. Since $x\in B_{1+9^{-1}}(0)$ and $r_{x}<1/10$, we see that $U_{\epsilon}$ satisfies Theorem \ref{the4.1} at scale $B_{r_{x}}(x)$. However, from ($\text{C}_{i-1}$) we now have
\[\mathop{\sup}\limits_{y\in B_{2r_{x}}(x)} \mathbf{\Theta}_{U_{\epsilon}}^{\epsilon}(2r_{x}, y)\leq M-\frac{(i-1)\eta}{2}.\]
Therefore, Theorem \ref{the4.1} (C) for  $\mathcal{U}_{i, x}$ becomes
\begin{equation}\label{prothemain1.2}
\sup_{z\in B_{2r_{y}}(y)}\mathbf{\Theta}_{U_{\epsilon}}^{\epsilon}(2r_{y},z)\leq M-\frac{i\eta}{2}
\end{equation}
for each $y\in\mathcal{U}_{i, x}$ with $r_{x}>r$. Theorem \ref{the4.1} (A) yields
\begin{equation}\label{prothemain1.3}
\big\{|U_{\epsilon}(x,0)|\leq 1-\tau\big\}\cap B_{r_{x}}(x)\subset\mathop{\cup}_{y\in\mathcal{U}_{i, x}}B_{r_{y}}(y)
\end{equation}
and Theorem \ref{the4.1} (B) implies
\begin{equation}\label{prothemain1.4}
\sum\limits_{y\in\mathcal{U}_{i, x}}r_{y}^{n-1}\leq c(n, s, \tau, \Lambda_{0}, \Lambda_{1}, W)r_{x}^{n-1}.
\end{equation}
Moreover, from the construction of Theorem \ref{the4.1} and ($\text{D}_{i-1}$), we have
\begin{equation}\label{prothemain1.5}
\sup\limits_{y\in\mathcal{U}_{i, x}}r_{y}\leq 10^{-1}r_{x}\leq 10^{-i}.
\end{equation}
We then set
\[\mathcal{U}_{i}=\{x\in\mathcal{U}_{i-1}:r_{x}=r\}\cup\mathop{\bigcup}_{\{x\in\mathcal{U}_{i-1}:r_{x}>r\}}\mathcal{U}_{i, x}.\]
From \eqref{prothemain1.2}, \eqref{prothemain1.3}, \eqref{prothemain1.4} and \eqref{prothemain1.5}, the set $\mathcal{U}_{i}$ satisfies the hypotheses ($\text{A}_{i}$), ($\text{B}_{i}$), ($\text{C}_{i}$) and ($\text{D}_{i}$). This completes the construction of the covering  and finishes the proof of Theorem \ref{themain1}.
\end{proof}
\begin{proof}[Proof of Theorem \ref{the perimeter estimate}]
The proof of Theorem \ref{the perimeter estimate} relies on Lemma \ref{lem Monotonicityformula}, \cite[Theorem 6.7]{Millot-Sire-Wang2019} and Theorem \ref{theRec2}. Since the
strategy of proof Theorem \ref{the perimeter estimate} is the same with the proof of Theorem \ref{thm volume estimate}, we will omit the details.
\end{proof}
\section{Proof of Proposition \ref{pro potential estimate} }
In this section, we give the proof of Proposition \ref{pro potential estimate}.
\begin{proof}[Proof of Proposition \ref{pro potential estimate}]
Choose a structural constant $\mathbf{\delta}_{W}\in (0, 1/2]$ such that
\[W''(t)\geq \frac{1}{2}\min\{W''(1), W''(-1)\}>0,\quad\text{for}~\big|1-|t|\big|\leq \mathbf{\delta}_{W}.\]
There are two cases we need to consider separately.

\textbf{Case 1.} $\big|1-|u_{\epsilon}|\big|\leq \mathbf{\delta}_{W}$ on $B_{2}(0).$

By \cite[Lemma 4.11]{Millot-Sire-Wang2019}, we have
\[W(u_{\epsilon})\leq c(n, s,\tau, \Lambda_{0}, \Lambda_{1}, W)\epsilon^{4s},\quad\text{on}~B_{1}(0).\]
Then it follows easily that  Proposition \ref{pro potential estimate} holds.

\textbf{Case 2.} $\big\{\big|1-|u_{\epsilon}|\big|>\mathbf{\delta}_{W}\big\}\cap B_{2}(0)\neq\emptyset.$

Take two positive constants $\mathbf{k}=\mathbf{k}(n, s, \delta_{W},\Lambda_{0}, \Lambda_{1},W), c=c(n, s, \delta_{W},\Lambda_{0}, \Lambda_{1},W)$ such that Theorem \ref{thm volume estimate} holds for $\tau=1-\delta_{W}$. Note that if $\mathbf{k}\epsilon\geq 1$, then it is trivial that Proposition \ref{pro potential estimate} holds. Therefore, we need only to consider the case $\mathbf{k}\epsilon<1$.
We get from Theorem \ref{thm volume estimate}  that if $r\in (\mathbf{k}\epsilon, 2)$, then
\begin{equation}\label{volume estimate on B2}
\mathscr{L}^{n}\Big(\mathcal{J}_{r}\big(\{|u_{\epsilon}|<1-\mathbf{\delta}_{W}\}\cap B_{2}(0)\big)\Big)\leq cr.
\end{equation}

Fix a positive constant $\nu<1$. For each $\epsilon$ such that $\mathbf{k}\epsilon<1$, choose an integer $p_{1}(\epsilon)$ such that $\nu^{p_{1}(\epsilon)-1}\in [\mathbf{k}\epsilon, \nu^{-1}\mathbf{k}\epsilon]$. Then
\[
\begin{aligned}
\int_{B_{1}(0)}W(u_{\epsilon})dx\leq &\int_{\mathcal{J}_{\mathbf{k}\epsilon}(\{|u_{\epsilon}|<1-\mathbf{\delta}_{W}\}\cap B_{1}(0))}W(u_{\epsilon})dx\\
&+\sum_{j=0}^{p_{1}(\epsilon)-2}\int_{\mathcal{A}_{j}}W(u_{\epsilon})dx+\int_{B_{1}(0)\backslash \mathcal{J}_{1}(\{|u_{\epsilon}|<1-\mathbf{\delta}_{W}\})}W(u_{\epsilon})dx.
\end{aligned}\]
where
\[\mathcal{A}_{j}:=\mathcal{J}_{\tau^{j}}\big(\{|u_{\epsilon}|<1-\mathbf{\delta}_{W}\}\cap B_{1}(0)\big)\backslash \mathcal{J}_{\tau^{j+1}}\big(\{|u_{\epsilon}|<1-\mathbf{\delta}_{W}\}\cap B_{1}(0)\big).\]
For $j=0, 1, \cdots, p_{1}(\epsilon)-2$ , we know from \cite[Lemma 4.11]{Millot-Sire-Wang2019} that
\[W(u_{\epsilon})\leq c\epsilon^{4s}\nu^{-4(j+1)s},\quad\text{on}~\mathcal{A}_{j},\]
where $c=c(n, s, \delta_{W},\Lambda_{0}, \Lambda_{1},W)$. Plugging this into \eqref{volume estimate on B2}, we obtain
\[\int_{B_{1}(0)}W(u_{\epsilon})dx\leq c\big(\epsilon+\sum_{j=0}^{p_{1}(\epsilon)-2}\epsilon^{4s}\nu^{(j+1)(1-4s)}+\epsilon^{4s}\big).\]
\begin{itemize}
\item If $s<1/4$, then
\[\sum_{j=0}^{p_{1}(\epsilon)-2}\nu^{(j+1)(1-4s)}<\infty.\]
Hence $\int_{B_{1}(0)}W(u_{\epsilon})dx\leq c\epsilon^{4s}.$
\item If $s=1/4$, we know from the definition of $p_{1}(\epsilon)$ that
\[p_{1}(\epsilon)\leq c|\log\epsilon|,\]
then \[\int_{B_{1}(0)}W(u_{\epsilon})dx\leq c\epsilon|\log\epsilon|.\]
\item If $s>1/4$, we have
\[\sum_{j=0}^{p_{1}(\epsilon)-2}\nu^{(j+1)(1-4s)}\leq c\nu^{p_{1}(\epsilon)(1-4s)}\leq c\epsilon^{1-4s}.\]
In this case,  we conclude that \[\int_{B_{1}(0)}W(u_{\epsilon})dx\leq c\epsilon.\]
\end{itemize}
Combining the above discussions together, we can get the desired estimate.
\end{proof}
\bibliography{Stratification}
\bibliographystyle{plain}

\end{document}